\documentclass[oneside,reqno]{amsart}
\usepackage{inputenc,amsfonts,amsmath,amssymb,mathtools}
\usepackage[all]{xy}
\usepackage{amsthm}
\usepackage[left=3.5cm, right=3.5cm, top=2cm]{geometry}
\usepackage{extarrows}
\usepackage{graphicx,subcaption}
\usepackage{wrapfig}
\usepackage{multirow}
\usepackage{tikz-cd}

\usepackage{xcolor,hyperref}
\definecolor{darkred}{rgb}{.8,0,0}
\definecolor{tocolor}{rgb}{.1,.1,.1}
\definecolor{urlcolor}{rgb}{.2,.2,.6}
\definecolor{linkcolor}{rgb}{.1,.1,.5}
\definecolor{citecolor}{rgb}{.4,.2,.1}
\definecolor{gray}{rgb}{.8,.8,.8}

\hypersetup{
	backref=true,
	colorlinks=true,
	urlcolor=urlcolor,
	linkcolor=linkcolor,
	citecolor=citecolor,
	pdfauthor = {Mykola Matviichuk and Brent Pym and Travis Schedler},
	pdftitle = {\title}
	}

\newcommand{\mapdef}[5]{\begin{array}{ccccc}
#1 &:& #2 &\rightarrow& #3 \\
&& #4 &\mapsto& #5
\end{array}}

\usepackage{aliascnt}
\newcommand{\thdef}[2]{
	\newaliascnt{#1}{theorem}  
	\newtheorem{#1}[#1]{#2}
	\aliascntresetthe{#1}  
	\newtheorem*{#1*}{#2}
	\expandafter\newcommand\expandafter{\csname #1autorefname\endcsname}{#2}
}

\newcommand{\edge}{\,
	\begin{tikzpicture}
		\draw (0,.5ex) -- (2ex,.5ex);
		\draw[fill] (0,.5ex) circle (.2ex);
		\draw[fill] (2ex,.5ex) circle (.2ex);
		\draw[fill] (0,0) circle (0);
	\end{tikzpicture}\,
}

\usepackage{tikz-cd,pgfplots}
\usetikzlibrary{angles,quotes}
\title{Creating  quantum projective spaces by deforming $q$-symmetric algebras}

\author{Mykola Matviichuk}
\address{Max Planck Institute for Mathematics, Bonn, Germany}
\email{mykola.matviichuk@gmail.com}
\author{Brent Pym}
\address{McGill University, Montreal, Canada}
\email{brent.pym@mcgill.ca}
\author{Travis Schedler}
\address{Imperial College London, UK}
\email{t.schedler@imperial.ac.uk}
\date{}

\newtheorem{theorem}{Theorem}[section]
\thdef{thm}{Theorem}

\thdef{lemma}{Lemma}

\thdef{proposition}{Proposition}

\thdef{corollary}{Corollary}

\thdef{conjecture}{Conjecture}
\thdef{problem}{Problem}
\theoremstyle{definition}

\thdef{definition}{Definition}

\thdef{examplex}{Example}

\thdef{remarkx}{Remark}

\thdef{theoremMain}{Main Theorem}

\thdef{theoremAlph}{Theorem}

\thdef{definitionAlph}{Definition}

\thdef{conjectureAlph}{Conjecture}

\usepackage{todonotes}

\newcommand{\HH}{\mathsf{HH}}
\newcommand{\HP}{\mathsf{HP}}

\newcommand{\ii}{\sqrt{-1}}
\newcommand{\n}{n}

\newcommand{\smoothedge}[5][blue]{
\begin{scope}
\def\Ea{#2}
\def\Eb{#3}
\def\Ra{#4}
\def\Rb{#5}
\draw[very thick,#1] (\Ea) -- (\Eb);
\foreach \x in {\Ra,\Rb} {
		\begin{pgfonlayer}{background layer}
		\begin{scope}
			\clip (\Ea) -- (\x) -- (\Eb);
			
			\draw [draw,opacity=0.5,fill,#1] (\x) circle (0.5) ;
		\end{scope}
		\end{pgfonlayer}
}
\end{scope}
}

\newcommand{\ps}{\lambda}

\newcommand{\ac}{\scriptstyle \textrm{!`}}
\newcommand{\bC}{\mathbb{C}}
\newcommand{\bCx}{\mathbb{C}^\times}
\newcommand{\bZ}{\mathbb{Z}}
\newcommand{\Mat}{\operatorname{M}}
\newcommand{\Ext}{\operatorname{Ext}}
\newcommand{\EE}{\operatorname{EExp}}
\newcommand{\bk}{k}
\newcommand{\bkx}{\bk^\times}
\newcommand{\sfV}{\mathsf{V}}
\newcommand{\sfA}{\mathsf{A}}
\newcommand{\sfB}{\mathsf{B}}

\newcommand{\bP}{\mathbb{P}}

\newcommand{\askew}[2][n]{\mathrm{AM}_{#1}({#2})}
\newcommand{\mskew}[2][n]{\mathrm{AM}_{#1}({#2}^\times)}

\newcommand{\rbrac}[1]{\left(#1\right)}
\newcommand{\abrac}[1]{\left\langle#1\right\rangle}
\newcommand{\set}[2]{\left\{#1\,\middle|\,#2\right\}}

\newcommand{\FO}[3][]{\mathrm{FO}(#2,#3)_{#1}}
\newcommand{\Ept}{z}
\newenvironment{example}
  {\pushQED{\qed}\examplex}
  {\popQED\endexamplex}

\newenvironment{remark}
  {\pushQED{\qed}\remarkx}
  {\popQED\endremarkx}

\numberwithin{equation}{section}

\begin{document}

\maketitle

\begin{abstract}
   We construct a large collection of ``quantum projective spaces'', in the form of Koszul, Calabi--Yau  algebras with the Hilbert series of a polynomial ring.  We do so by starting with the toric ones (the $q$-symmetric algebras), and then deforming their relations using a diagrammatic calculus, proving unobstructedness of such deformations under suitable nondegeneracy conditions.  We then prove that these algebras are identified with the canonical quantizations of corresponding families of quadratic Poisson structures, in the sense of Kontsevich.  In this way, we obtain the first broad class of quadratic Poisson structures for which his quantization can be computed explicitly, and shown to converge, as he conjectured in 2001.  
\end{abstract}

\setcounter{tocdepth}{1}

\section{Introduction}

\subsection{Non-commutative deformations of polynomial algebras}

It is a longstanding central question in algebra and mathematical physics to understand \emph{quantizations} of a given variety $X$, i.e.~non-commutative, but associative, deformations of the (homogeneous) coordinate ring, or of related objects such as various categories of sheaves, depending on some parameter $\hbar$.  When $X$ is smooth affine, formal quantizations (given by formal power series in $\hbar$) are equivalent to Poisson brackets by the groundbreaking work of Kontsevich~\cite{Kontsevich2003};  while his proof provides an ``explicit'' quantization in the form of a Feynman expansion, it is difficult to directly assess its convergence, or extract concrete formulae for the deformed algebras from knowledge of the corresponding Poisson bracket. 

When $X= \bP^{n-1}$ is a projective space, all quantizations are given by homogeneous quantizations of the homogeneous coordinate ring  $\bC[X]=\bC[x_0,\ldots,x_{n-1}]$. Even in this seemingly simple case, the problem of constructing and classifying these quantizations remains open and is poorly understood for $n \geq 5$. However, there is a general understanding of the sort of rings one should look for: one should replace the usual relation that the variables commute, $x_ix_j=x_jx_i$, with more general quadratic relations $x_ix_j=x_jx_i+\sum_{kl}c_{ij}^{kl}(\hbar)x_kx_l$ where $\hbar$ is a parameter, taking care to ensure that the algebra remains of the same size (i.e.~the deformation is flat). The result should enjoy many of the same properties as the polynomial ring, reflecting the geometric properties of projective space.  In particular, it should be a Koszul algebra satisfying the celebrated Artin--Schelter regularity condition, or what is basically equivalent, Ginzburg's (twisted) Calabi--Yau condition; these are homological conditions on a ring that correspond geometrically to the fact that the affine space $\mathbb{A}^{n}$ is smooth and has trivial canonical bundle.  Moreover, Kontsevich conjectures in \cite{Kontsevich2001} that in this case, his canonical quantization recipe should result in algebras for which the formal power series $c_{ij}^{kl}(\hbar)$ determining the relations of his formal quantizations are convergent, i.e.~define analytic functions of $\hbar$ in some disk around $\hbar =0$.

For $n \leq 3$, there is a complete classification of Artin--Schelter regular algebras in terms of geometric data~\cite{Artin1990,Artin1987} (including algebras that do not arise from quantization).  For $n=4$, the quantizations were classified into six irreducible families \cite{Pym2015} using the corresponding classification of Poisson brackets~\cite{Cerveau1996,Loray2013}.  In higher dimension, various examples are known, most notably the elliptic algebras of Feigin--Odesskii~\cite{Feigin1998,Odesskii1989} whose Hilbert series and regularity properties have been established (for most values of the parameters) in \cite{Chirvasitu2021,Tate1996}.  However, even for $n=5$, the classification is currently out of reach.

Our goal in this paper is to produce a large class of \emph{explicit, analytic} quantizations of $\bP^{n-1}$ for $n \ge 3$ odd, by quantizing Poisson structures satisfying a suitable nondegeneracy condition.  These quantizations are Artin--Schelter regular algebras with the correct Hilbert series, given by explicit quadratic relations defined combinatorially via certain decorated graphs---a special case of the ``smoothing diagrams'' for log symplectic manifolds which we introduced in \cite{Matviichuk2020} and which were further developed by the first author in \cite{Matviichuk2023b}.  By construction, these quantizations form an open subset of the space of possible quadratic relations with correct Hilbert series, up to the action of $\mathsf{GL}(n)$, and this gives an effective description of many new irreducible components of this space---more precisely, a classification of the algebras that admit a suitably generic toric degeneration. The families of non-commutative $\bP^{n-1}$ we can construct are indexed by the same diagrams as the families of log symplectic structures on $\bP^{n-1}$ admitting toric degenerations studied in \cite{Matviichuk2020}, in particular, there are the same number of each.  Already for $n=5$, this gives a lower bound on the number of irreducible components of the moduli space of quantum $\mathbb{P}^4$: there are at least 40.  Moreover, leveraging the calculation of Kontsevich's quantization for toric Poisson structure via Hodge theory in \cite{Kontsevich2008,Lindberg2021,Lindberg2024}, we prove that when the smoothing diagram has no cycles, our algebras are exactly the canonical quantizations in the sense of Kontsevich, and as a consequence we verify his conjecture on convergence (up to isomorphism) of the canonical quantization of quadratic Poisson structures in these cases.

\subsection{$q$-symmetric algebras}
Our quantizations are constructed in two steps.  The first is well known: we make an ``easy'' deformation of the polynomial ring, which retains its toric symmetry.  More precisely, given a matrix $q=(q_{ij})_{i,j=0}^{n-1}$ of non-zero complex numbers such that $q_{ii}=1$ for all $i$, and $q_{ji}=q_{ij}^{-1}$ for $i < j$, we consider the non-commutative \emph{$q$-symmetric algebra} (a.k.a. ``$q$-polynomial algebra'', ``skew polynomial ring'', ``quantum polynomial ring'', ``quantum affine space'', ...) $\mathsf{A}_q$ generated by $x_0, \ldots, x_{n-1}$
modulo the relations
\begin{align*}
x_i x_j  &= q_{ij} x_j x_i  & 0\le i,j \le n-1.
\end{align*}
It has a natural action of the torus $(\bCx)^n$ by rescaling the generators.

For any choice of $q$, the algebra $\mathsf{A}_q$ is Koszul, twisted Calabi--Yau and Artin--Schelter regular. Moreover, it is (untwisted) Calabi--Yau if and only if $q$ is normalized, in the sense that $\prod_{j=0}^{n-1} q_{ij} = 1$ for all $i$.  These algebras can be viewed as quantizations of the toric Poisson algebra $\sfB_\lambda$, given by the polynomial ring $\bC[x_0,\ldots,x_{n-1}]$ with the Poisson bracket
\begin{align*}
\{x_i,x_j\} &= \ps_{ij} x_i x_j & 0\le i,j \le n-1,
\end{align*} 
via the correspondence $q_{ij} := e^{\hbar \ps_{ij}}$.  As proven in \cite{Kontsevich2008,Lindberg2021,Lindberg2024}, this is exactly the canonical quantization in the sense of Kontsevich for these particular Poisson brackets.

\subsection{Main result}
The second step is the main content of the paper: we take a $q$-symmetric algebra $\sfA_q$, with normalized $q$, and determine its possible deformations as a quadratic algebra, under an explicit Zariski-open genericity condition on $q$, namely that $q_{ij} = e^{\lambda_{ij}}$ for a skew-symmetric corank-one matrix $(\lambda_{ij})$ lying in the complement of an explicit collection of hyperplanes (see \autoref{def:mgeneric}).  Note that this requires $n$ to be odd, so that the corresponding projective space $\bP^{n-1}$ is even-dimensional.  The Poisson bracket of $\sfB_\lambda$ then induces a toric log symplectic structure on $\bP^{n-1}$.

The genericity condition translates into the statement that the $\bCx$-invariant Hochschild cohomology of $\sfA_q$ coincides with the $\bCx$-invariant Poisson cohomology of $\sfB_\lambda$, so that we may express the deformations of $\sfA_q$ in terms of the ``smoothing diagrams'' mentioned above.   More precisely, the smoothing diagram of a given $q$-symmetric algebra $\sfA_q$ is a complete graph on $n$ vertices, with certain edges and angles colored. The colored edges are in one-to-one correspondence with the non-toric infinitesimal deformations of $\mathsf{A}_q$, i.e.~the isotypical components of the deformation space with nonzero torus weight.  The colored angles then encode the corresponding torus weights; see \autoref{subsec:smoothingDiagrams} for further details.  

The colored edges in the smoothing diagram turn out to decompose as a union of disjoint cycles and chains.  By controlling the torus weights of the possible obstructions to deformations, we show that the chains extend to filtered deformations of $\sfA_q$ that are polynomial in the deformation parameter, and unique up to isomorphism.  Meanwhile, using the first author's classification of smoothing diagrams with cycles~\cite{Matviichuk2023b}, and some asymptotics of theta functions, we show that cycles extend to formal deformations given by the elliptic algebras of Feigin--Odesskii~\cite{Feigin1998,Odesskii1989}.  Finally, we show that the interaction between such deformations are governed by a braiding of the tensor product determined by $q$.  In this way, we arrive at the following result.
\begin{theoremAlph}[see \autoref{thm:DeformCyclesAndChains}] \label{thm:Main}
\textit{
If $q$ is generic in the above sense, then the non-toric infinitesimal deformations of $\mathsf{A}_q$ are jointly unobstructed. More precisely, every such infinitesimal deformation extends to an analytic family of algebras, given by a braided tensor product of the filtered deformation obtained from the chains in the smoothing diagram, and a Feigin--Odesskii elliptic algebra associated to each cycle.}
\end{theoremAlph}

\subsection{An example}
Let $n=5$ and consider the $n\times n$ matrix $q_{ij}=e^{\hbar \ps_{ij}}$, where $\hbar \in \bC$ and

$$
(\ps_{ij})_{i,j=0}^4 =
\begin{pmatrix}
0 & -6 & 6 & -2 & 2 \\
6 & 0 & -3 & -1 & -2 \\
-6 & 3 & 0 & -3 & 6 \\
2 & 1& 3 & 0 & -6 \\
-2 & 2 & -6 & 6 & 0
\end{pmatrix}.
$$

The smoothing diagram of $\mathsf{A}_q$ can be calculated according to the recipe described in \autoref{subsec:smoothingDiagrams}, and looks as follows:
$$
\includegraphics[scale=0.6]{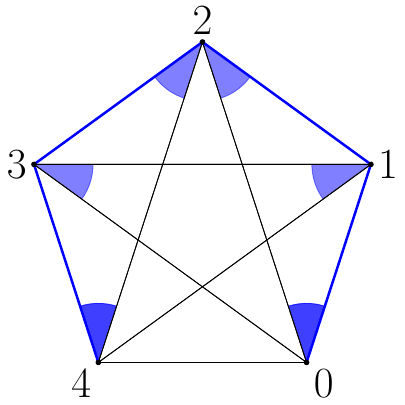}
$$

The four colored edges represent a four-dimensional space of non-toric infinitesimal deformations, given by deforming four of the quadratic relations  $x_ix_j = q_{ij}x_jx_i$ as follows:
$$
\begin{matrix}
x_1x_0 &=& q_{10}~x_0x_1 &+&\varepsilon\gamma_1~x_2x_3,\\
x_2x_1 &=& q_{21}~x_1x_2&+&\varepsilon\gamma_2~x_0^2, \\
x_3x_2 &=& q_{32}~x_2x_3 & +&\varepsilon\gamma_3~ x_4^2,\\
x_4x_3 &=& q_{43}~x_3x_4&+&\varepsilon\gamma_4~x_1x_2,
\end{matrix}
$$
where $\gamma_i\in \mathbb{C}$ are arbitrary fixed constants and $\varepsilon$ is the deformation parameter. \autoref{thm:Main}  implies that for generic $\hbar\in\mathbb{C}$, there is a way to add terms of higher order in $\varepsilon$ to the relations, so that the resulting quadratic algebra is a formal flat deformation of $\mathsf{A}_q$; these are worked out explicitly in \autoref{ex:4chain}.

\subsection{Homological properties}
Having constructed filtered deformation of $\mathsf{A}_q$, we study its Koszul resolution via superpotentials, and prove the following:

\begin{theoremAlph} [see Propositions \ref{prop:deformedKoszulNoCycle} and \ref{prop:deformedCYNoCycle}] {\it
Let $\mathsf{A}_{q,I}$ be the deformation of $\mathsf{A}_q$ corresponding to a cycle-free subset $I$ of the set of colored edges in its smoothing diagram.  Then  $\mathsf{A}_{q,I}$ is Koszul, Calabi--Yau and Artin--Schelter regular.}
\end{theoremAlph}
Combined with the results of \cite{Chirvasitu2021} concerning the Feigin--Odesskii algebras, we deduce that for an arbitrary smoothing diagram, the same properties hold for all but at most countably many values of the deformation parameter.

\subsection{Application to Kontsevich's conjecture on deformation quantization}
As we mentioned already, the algebra $\mathsf{A}_q$, for $q_{ij}=e^{\ps_{ij}}$, is Kontsevich's canonical quantization of the Poisson bracket $\{x_i,x_j\} = \lambda_{ij}x_ix_j$.   In \autoref{sec:kontsevich}, we identify the filtered deformations of $\mathsf{A}_q$, which we explicitly constructed, with the canonical quantizations of the corresponding filtered deformations of the Poisson bracket.
\begin{theoremAlph}[see \autoref{thm:kontsevich}]\emph{
If a Poisson structure on $\bP^{n-1}$ admits a filtered log symplectic toric degeneration, then its canonical quantization is gauge equivalent to one for which the star products of the coordinate functions are convergent power series. }
  \end{theoremAlph}
This shows that Kontsevich's 2001 conjecture~\cite[Conjecture 1]{Kontsevich2001} holds up to isomorphism of algebras for this class of Poisson structures.  To our knowledge, this is the first broad class of quadratic Poisson structures  for which this has been verified.

\subsection{How to generalize to arbitrary fields}
In this paper, for readability, after some preliminaries that are insensitive to the ground field $k$, we restrict to algebras over $k=\bC$, from \autoref{sec:ee} through \autoref{sec:CY}. For our applications to Kontsevich's conjecture in \autoref{sec:kontsevich}, we extend to the setting of power and Laurent series.

However, all of the results, except for Theorem \ref{thm:good-q} and the results of \autoref{subsec:FOcycle} and \autoref{sec:kontsevich}, can be generalized, mostly without change, to arbitrary fields $k$ of characteristic zero.  To do this, one can either work with a subring $R \subseteq k$ where the exponential function $\exp: R \to k^\times$ makes sense (in particular taking sums to products), or relax the condition $q=\EE(\ps)$ with $\ps$ generic to merely requiring that the sum-zero weights of $\HH(\mathsf{A}_q)$ be a subset of the weights of the corresponding Poisson cohomology groups $\HP(\mathsf{B}_\ps)$ for some $\ps$ of corank one,
along with the requirement that $q$ be normalized. 

\subsection{Acknowledgements}
We thank Ryo Kanda for helpful correspondence regarding the results of \cite{Chirvasitu2021}. Thanks to Colin Ingalls and Sarah Witherspoon for pointing out the article \cite{Gr16}, and Pavel Etingof for useful discussions.

This project was conceived during the authors' participation in the Oberwolfach Research Fellows program in November 2023. We are grateful to the MFO staff for their hospitality and for creating an excellent research environment. M.M.~thanks the Max Planck Institute for Mathematics for their hospitality.  B.P.~was supported by a faculty startup grant at McGill University, a New university researchers startup grant from the Fonds de recherche du Qu\'ebec -- Nature et technologies (FRQNT), and by the Natural Sciences and Engineering Research Council of Canada (NSERC), through Discovery Grant RGPIN-2020-05191.

\section{$q$-symmetric algebras and their Hochschild cohomology}

\subsection{Multiplicatively alternating matrices}
Let $\bk$ be a field of characteristic zero, and let $\bkx = \bk\setminus\{0\}$ be the multiplicative group of invertible elements.

\begin{definition}A \textit{multiplicatively alternating matrix}  (of size $n$)\footnote{The terminology is justified by the fact that 
the $\bZ$-bilinear form $(~,~):\mathbb{Z}^n \times \mathbb{Z}^n \to \bk^\times$ defined by $({\bf u},{\bf v})=\prod_{i,j} q_{ij}^{u_iv_j}$ is alternating, i.e. $({\bf u},{\bf u})=1$, for all ${\bf u}\in\mathbb{Z}^n$} is a matrix $q \in \Mat_n(\bk)$, whose entries $q_{ij}$ are all invertible, and such that $q_{ii}=1$ for all $i$ and $q_{ij}=q_{ji}^{-1}$ for all $i,j$.  We denote by $\mskew{\bk}$ the set of all such matrices.
\end{definition}

\subsection{$q$-symmetric algebras}

Let $\sfV = k^n$ with linear coordinates  $x_0, \ldots, x_{n-1} \in \sfV^*$. We denote the dual basis by $\partial_0, \ldots, \partial_{n-1} \in \sfV$.  Given a multiplicatively alternating matrix $q \in \mskew{\bk}$, the $q$-symmetric algebra $\mathsf{A}_q=\mathsf{S}_q(\mathsf{V}^*)$ is the associative $\bk$-algebra generated by $x_0, \ldots, x_{n-1}$
modulo the relations
$$
x_i x_j  = q_{ij} x_j x_i.
$$
Let $R \subseteq \mathsf{V}^* \otimes \mathsf{V}^*$ be the span of the relations $x_i x_j - q_{ij} x_j x_i$.

We denote the quadratic dual algebra by $\mathsf{A}_q^!=\bigwedge_q(\mathsf{V})$, with product denoted $\wedge_q$.  Concretely, $\mathsf{A}_q^!$ is the algebra generated by $\partial_0, \ldots, \partial_{n-1}$, with defining relations
$$
\partial_i \wedge_q \partial_j = -q_{ji} ~\partial_j \wedge_q \partial_i.
$$

\begin{remark} The terminology ``$q$-symmetric algebra'' can be justified as $\mathsf{A}_q$ can be interpreted as a symmetric algebra in the symmetric monoidal category of $\bZ^n$-graded vector spaces with braiding given on homogeneous elements by $x \otimes y \mapsto \prod_{i,j} q_{ij}^{|x|_i |y|_j}$, for $|x|$ the multidegree of $x$ and $|x|_i$ the $i$-th component.   Namely, it is the symmetric algebra generated by the standard $n$-dimensional vector space which is a direct sum of one-dimensional vector spaces of each degree $e_i=(0,\ldots,0,1,0,\ldots,0)$.   
Also, from this point of view, the quadratic dual $\mathsf{A}_q^!$ is an exterior algebra (or, if using a sign on the braiding, then it can also be viewed as a symmetric algebra on odd generators).
  \end{remark}

\subsection{Torus action} 
The torus $(\bkx)^n$ acts on $\sfV=k^n$ by rescaling the coordinates $x_i$.  This action preserves the defining relations of $\mathsf{A}_q$, and thus induces dual actions by algebra automorphisms of $\mathsf{A}_q$ and $\mathsf{A}_q^!$, decomposing them into weight spaces
\begin{align*}
\mathsf{A}_q &= \bigoplus_{\mathbf{w}\in\bZ^n} \mathsf{A}_q^{\mathbf{w}} & \mathsf{A}_q &= \bigoplus_{\mathbf{w}\in\bZ^n} (\mathsf{A^!}_q)^{\mathbf{w}}
\end{align*}
Note that each weight space is at most one-dimensional, e.g.~for $\mathsf{w}=(w_0,\ldots,w_{n-1}) \in \bZ^n$, the subspace $\mathsf{A}_a^{\mathbf{w}}$ is spanned by the monomial $x_0^{w_0}\cdots x_{n-1}^{w_{n-1}}$ if $w_i \ge 0$ for all $i$, and is zero otherwise.

\subsection{The Koszul resolution} 
Let  $\mathsf{A}_q^{\ac} \subseteq T\mathsf{V}^*$ be the ``quadratic dual coalgebra'', i.e.~the graded linear dual of $\mathsf{A}_q^!$.  It is explicitly given as the subspace
\[
  \mathsf{A}_q^{\ac} = k \oplus \sfV^* \oplus \bigoplus_{m \geq 2} \bigcap_{0 \leq i \leq m-2} (\sfV^*)^{\otimes i} \otimes R \otimes (\sfV^*)^{\otimes (m-i-2)}.
  \]
We equip $\mathsf{A}_q^{\ac}$ with the grading by tensor degree, so $\mathsf{A}_q^{\ac} = \bigoplus_{m \geq 0} (\mathsf{A}_q^{\ac})_m$. The elements of $\mathsf{A}_q^{\ac}$ are linear combinations of $d_{i_1} \wedge_q d_{i_2} \wedge_q ... \wedge_q d_{i_m}$, where $d_i := x_i$, and $d_i \wedge_q d_j := d_i \otimes d_j - q_{ij} d_j \otimes d_i$.

Consider the complex $K'=K'(\mathsf{A}_q)$ of left $\mathsf{A}_q$-modules given by
$$
\xymatrix{
\mathsf{A}_q & \ar[l] \mathsf{A}_q \otimes (\mathsf{A}_q^{\ac})_1 & \ar[l] \mathsf{A}_q \otimes (\mathsf{A}_q^{\ac})_2 & \ar[l] \dots,}
$$
with differential given by the splitting map:
$$\delta(f \otimes (d_{i_1} \otimes \cdots \otimes d_{i_m})) = f x_{i_1} \otimes (d_{i_2} \otimes \cdots \otimes d_{i_m}).
$$

We view $\bk$ as a left $\mathsf{A}_q$-module, where the generators $x_i$ all act by zero. The following result is standard:
\begin{lemma}\label{lm:qPolyKoszul}
The complex $K'(\mathsf{A}_q)$ is a minimal free resolution of the left $\mathsf{A}_q$-module $\bk$. In other words, the algebra $\mathsf{A}_q$ is Koszul.
\end{lemma}

\begin{proof}
  As a complex of vector spaces, $K'(\mathsf{A}_q)$ is isomorphic to the tensor product of complexes $K'(\mathbb{C}[x_i])=\big(\mathbb{C}[x_i] \longleftarrow \mathbb{C}[x_i] \otimes d_i 
  \big)$, with map given by multiplication. The isomorphism is given by $q$-antisymmetrization of the $x_i$ factors, that is,
  \[
    d_{i_1} \otimes \cdots \otimes d_{i_m} \mapsto \sum_{\sigma \in S_m} \text{sign}(\sigma) \sigma \cdot d_{i_1} \otimes \cdots \otimes d_{i_m}, 
  \]
  using the action of $S_m$ on tensors generated by
  \[
    (j,j+1) \cdot d_{i_1} \otimes \cdots \otimes d_{i_m} = q_{i_j i_{j+1}} d_{i_1} \otimes \cdots \otimes d_{i_{j-1}} \otimes d_{i_{j+1}} \otimes d_{i_{j}} \otimes d_{i_{j+2}} \otimes \cdots \otimes d_{i_m}. \qedhere
    \]
\end{proof}

\subsection{Hochschild cohomology of $\mathsf{A}_q$}
In this section we give a simple, explicit description of the Hochschild cohomology of $\mathsf{A}_q$, highlighting the weights that occur therein.  This result is not new, and can be extracted from more general results in the literature, e.g., \cite[Theorem 3.3]{Gr16}.

Consider the complex $K=K(\mathsf{A}_q)$ of free $\mathsf{A}_q$-bimodules given by
$$
\xymatrix{
\mathsf{A}_q \otimes \mathsf{A}_q& \ar[l] \mathsf{A}_q \otimes (\mathsf{A}_q^{\ac})_1 \otimes \mathsf{A}_q& \ar[l] \mathsf{A}_q \otimes (\mathsf{A}_q^{\ac})_2 \otimes \mathsf{A}_q& \ar[l] \dots,}
$$
where the differential is
$$
\delta\big(f \otimes d_{i_1}  \wedge_q ... \wedge_q d_{i_m}\otimes g\big) = \sum_{j=1}^m (-1)^{j-1} q_{i_1i_j} q_{i_2i_j} ... q_{i_{j-1}i_j} f x_{i_j} \otimes d_{i_1}  \wedge_q ...~ \widehat{d_{i_j}}~... \wedge_q d_{i_m} \otimes g +
$$
$$
+ (-1)^m \sum_{j=1}^m (-1)^{m-j} q_{i_j i_{j+1}} q_{i_ji_{j+2}} ... q_{i_j i_m} f \otimes d_{i_1}  \wedge_q ...~ \widehat{d_{i_j}}~... \wedge_q d_{i_m} \otimes x_{i_j} g.
$$
Since the algebra $\mathsf{A}_q$ is Koszul, the complex $K$ is a minimal free resolution of the $\mathsf{A}_q$-bimodule $\mathsf{A}_q$. Hence the Hochschild cohomology $\mathsf{HH}(\mathsf{A}_q)$ can be computed by taking the cohomology of ${\rm Hom}(K, \mathsf{A}_q)$, where the hom is meant in the sense of $\mathsf{A}_q$-bimodules. The complex ${\rm Hom}(K, \mathsf{A}_q)$, whose elements we call \emph{$q$-polyvectors}, has the form
$$
\xymatrix{
\mathsf{A}_q \ar[r]&  \mathsf{A}_q \otimes (\mathsf{A}_q^!)_1 \ar[r]&  \mathsf{A}_q \otimes (\mathsf{A}_q^!)_2 \ar[r]&  \dots,}
$$
with the differential
$$
\mathsf{d}_q (h \otimes \partial_{i_1} \wedge_q ...\wedge_q \partial_{i_m}) = \sum_{j=0}^{n-1} x_j h \otimes \partial_j \wedge_q  \partial_{i_1} \wedge_q ...\wedge_q \partial_{i_m} - (-1)^m \sum_{j=0}^{n-1} h x_j \otimes  \partial_{i_1} \wedge_q ...\wedge_q \partial_{i_m} \wedge_q \partial_j.
$$
The $(\bkx)^n$-action on $\mathsf{A}_q$ induces an action on the Hochschild cohomology, giving a decomposition into weight spaces $\mathsf{HH}(\mathsf{A}_q) = \oplus_{\mathbf{w}}\mathsf{HH}(\mathsf{A}_q)^\mathbf{w}$.  We shall now give an explicit description of the nonzero weight spaces, as a function of the multiplicatively alternating matrix $q$.  To do so, we introduce the following terminology.

\begin{definition} \label{def:HHcontributing}
Let $q \in \mskew{\bk}$ be a multiplicatively alternating matrix of size $n$.  
A $(\bk^\times)^n$-weight $\mathbf{w}=(w_0,...,w_{n-1})\in\mathbb{Z}^{n}$ is called \textit{$q$-Hochschild-contributing} if the following conditions hold:
\begin{enumerate}
\item The inequality $w_i\ge -1$ holds for every $i$, and
\item The identity
\begin{align}
\prod_{j=0}^{n-1} q_{ij}^{w_j} = 1 \label{eq:HHcontrib}
\end{align}
holds for every $i$ such that $w_i\ge 0$.
\end{enumerate}
\end{definition}
The terminology is justified by the following result.
\begin{lemma}\label{lm:HHtoric}
The weight component $\mathsf{HH}(\mathsf{A}_q)^\mathbf{w}$ is non-zero if and only if $\mathbf{w}$ is Hochschild-contributing, in which case
$$
\mathsf{HH}(\mathsf{A}_q)^\mathbf{w} \cong \bk\rho \otimes {\textstyle \bigwedge_q} \bigoplus\limits_{j:w_j\ge0} \bk x_j\partial_j,
$$
where $\rho$ is the unique, up to a scalar, $q$-polyvector of weight $\mathbf{w}$ of the minimal degree.
\end{lemma}

\begin{proof}
For a $q$-polyvector $\tau = h \otimes \partial_{i_1}\wedge_q ... \wedge_q \partial_{i_m}$ of weight $\mathbf{w}$, we have
$$
\mathsf{d}_q (\tau) = \sum_{j=0}^{n-1} x_j h \otimes \partial_j \wedge_q  \partial_{i_1} \wedge_q ...\wedge_q \partial_{i_m} - (-1)^m \sum_{j=0}^{n-1} h x_j \otimes  \partial_{i_1} \wedge_q ...\wedge_q \partial_{i_m} \wedge_q \partial_j=
$$
$$
=\sum_{j=0}^{n-1} \left(1 - \prod_{k=0}^{n-1} q_{kj}^{w_k} 
\right)x_j h \otimes \partial_j \wedge_q  \partial_{i_1} \wedge_q ...\wedge_q \partial_{i_m}.
$$
Let us choose the linear isomorphism $\varphi$ between the weight $\mathbf{w}$-parts of $\mathsf{S}_q(\mathsf{V}^*) \otimes \bigwedge_q(\mathsf{V})$ and $\mathsf{S}(\mathsf{V}^*) \otimes \bigwedge(\mathsf{V})$ such that 
$$\varphi(x_{j_1} ... x_{j_s} \otimes \partial_{j_1} \wedge_q ... \wedge_q \partial_{j_s} \wedge_q \rho)=x_{j_1} ... x_{j_s} \otimes \partial_{j_1} \wedge ... \wedge \partial_{j_s} \wedge \rho,$$ 
whenever $j_1<...<j_s$ and $w_{j_1},...,w_{j_s}\ge0$. Then we have $\varphi (\mathsf{d}_q (\tau)) = \alpha \wedge \varphi(\tau)$, where 
\begin{equation} \label{eq:KoszulDiffAq}
\alpha = \sum_{j=0}^{n-1}\alpha_j~ x_j \otimes \partial_j, ~~~\alpha_j = 1 - \prod_{k=0}^{n-1} q_{kj}^{w_k} 
\end{equation}
If $\mathbf{w}$ is Hochschild-contributing, the differential $\alpha \wedge$ restricted to $\big(\mathsf{S}(\mathsf{V}^*) \otimes \bigwedge(\mathsf{V})\big)^\mathbf{w}$ is zero. Otherwise, there is at least one index $j$ with $w_j\ge 0$ and $\alpha_j\not=0$. Then the contraction operator
\begin{equation}\label{eq:homotopyKoszulDiff}
\dfrac{1}{\alpha_j x_j} \iota_{d_j}: \Big(\mathsf{S}^\bullet(\mathsf{V}^*) \otimes \wedge^\bullet(\mathsf{V})\Big)^\mathbf{w}\longrightarrow 
\Big(\mathsf{S}^{\bullet-1}(\mathsf{V}^*) \otimes \wedge^{\bullet-1}(\mathsf{V})\Big)^\mathbf{w}
\end{equation} 
defines the homotopy between the transported Koszul differential $\varphi \circ \mathsf{d}_q \circ \varphi^{-1}$ and the zero differential.
\end{proof}

\subsection{The Calabi--Yau condition}

Calabi--Yau algebras were defined by Ginzburg in \cite{Ginzburg2007} as a non-commutative generalization of Calabi--Yau manifolds. An algebra $\mathsf{A}$ is called $n$-\textit{Calabi--Yau} if  $\mathsf{A}$ admits a finite projective $\mathsf{A}$-bimodule resolution 
and $\Ext^\bullet_{\mathsf{A}-\text{bimod}}(\mathsf{A}, \mathsf{A} \otimes \mathsf{A}) \cong \mathsf{A}[-n]$ as an $\mathsf{A}$-bimodule. 

In the case of the algebras $\mathsf{A}_q$, the Calabi--Yau property can be expressed in terms of the following elementary condition on matrices.
\begin{definition}
A multiplicatively alternating matrix $q$ is \emph{normalized} if its rows (or equivalently, its columns) multiply to one, i.e.
\begin{align}
\prod_j q_{ij} = 1 \qquad \textrm{for all }i.
\end{align}
\end{definition}
Then by the calculations in \cite[Example 5.5]{Rogalski2014} or in \autoref{sec:CY} below, we have the following.
\begin{lemma}\label{lm:qPolyCY}
Let $q$ be a multiplicatively alternating matrix.  Then the algebra $\mathsf{A}_q$ is Calabi--Yau if and only if $q$ is normalized.
\end{lemma}

\subsection{Non-commutative projective spaces} 
Just as the polynomial ring $k[x_0,\ldots,x_{n-1}]$ is the homogeneous coordinate ring of the projective space $\mathbb{P}^{n-1}$, the algebras  $\mathsf{A}_q$ may be viewed as homogeneous coordinate rings for non-commutative analogues of $\mathbb{P}^{n-1}$.  More precisely, the category $\mathsf{qgr}(\mathsf{A}_q)$ of graded $\mathsf{A}_q$-modules modulo torsion modules is a flat deformation of the quasi-coherent sheaves on $\mathbb{P}^{n-1}$ over the parameter space $q \in \mskew{k}$.  However, different algebras $\mathsf{A}_q$ may have equivalent categories of graded modules.  For instance, we have the following:

\begin{lemma}
Let $q\in\mskew{\bk}$ be a multiplicatively alternating matrix, let  $\tau = (\tau_i)_i \in (\bkx)^n$ be an element of the torus (viewed as a diagonal matrix) and let $\tau q \tau^{-1}$ be the conjugation, i.e.~the multiplicatively alternating matrix with entries
\[
(\tau q \tau^{-1})_{ij} = \frac{\tau_i}{\tau_j} q_{ij}
\]
Then the algebras $\mathsf{A}_q$ and $\mathsf{A}_{\tau q\tau^{-1}}$ have equivalent categories of graded modules.
\end{lemma}
\begin{proof}
By a theorem of Zhang~\cite{Zhang1996}, two graded algebras have equivalent categories of graded modules if and only if one is isomorphic to a graded twist (aka Zhang twist) of the other.  In the case at hand of quadratic algebras  $\mathsf{A}_p = T(\sfV^*)/(R_p)$ and $\mathsf{A}_q = T(\sfV^*)/(R_q)$, such a twist is equivalent to a linear automorphism $\eta \in GL(\sfV^*)$ such that $(1 \otimes \eta)R_p = R_q$. Taking $\eta=\tau$, so that $\tau(x_i) = \tau_i x_i$, we have
\[
(1\otimes \tau)(x_i\otimes x_j - q_{ij} x_j \otimes x_i) = x_i \otimes (\tau_j x_j) - q_{ij} x_j \otimes (\tau_i x_i) = \tau_j\cdot \left(x_i\otimes x_j - \frac{\tau_i}{\tau_j}q_{ij} x_j \otimes x_i\right)
\]
so that $1 \otimes \tau$ sends a basis of $R_q$ to a basis of $R_{\tau q \tau^{-1}}$, as desired. 
\end{proof}

The following then shows that every category of the form $\mathsf{qgr}(\mathsf{A}_q)$, $q \in \mskew{\bk}$, is equivalent to one for which the algebra $\mathsf{A}_q$ is Calabi--Yau, provided the field $\bk$ admits $n$th roots.
\begin{lemma}
Suppose that every element of $\bk$ admits an $n$th root in $\bk$. Then for every multiplicatively alternating matrix $q\in\mskew{\bk}$, there exists a torus element $\tau \in (\bkx)^n$ such that $\tau q \tau^{-1}$ is normalized.
\end{lemma}

\begin{proof}
Let $\nu_i := \prod_{j}q_{ij}$ and let $\tau_i$ be any $n$th root of $1/\nu_i$.  Note that
\[
\rbrac{\prod_i \tau_i}^n = \prod_i \tau_i^n = \prod_i \nu_i^{-1} = \prod_i\prod_j q_{ji} = 1
\]
since $q$ is multiplicatively alternating.  Hence $\mu := \prod_i \tau_i$ is an $n$th root of unity.  Replacing $\tau_0$ with $\mu^{-1}\tau_0$ if necessary, we may assume without loss of generality that $\prod_i \tau_i = 1$.  We then have
\[
\prod_j (\tau q \tau^{-1})_{ij} = \prod_j \rbrac{\frac{\tau_i}{\tau_j} q_{ij}} =  \frac{ \tau_i^n \prod_j q_{ij}}{\prod_j \tau_j} = \frac{\nu_i^{-1} \prod_j q_{ij}}{1} = 1
\]
as desired.
\end{proof}

\section{Log canonical Poisson brackets and their Poisson cohomology}\label{sec:HP(A_pi)}

\subsection{Additively alternating matrices}
As in the previous section, $\bk$ denotes a field of characteristic zero.  We denote by $\askew{\bk} \subset \Mat_n(\bk)$ the set of alternating matrices in the classical sense, i.e.~matrices $\ps = (\ps_{ij})_{i,j=0}^{n-1} \in \Mat_{n}(\bk)$ such that $\ps_{ij} =-\ps_{ji}$ for all $i,j$.  When we wish to distinguish these from the multiplicatively alternating matrices above, we refer to them as \emph{additively alternating}.

\subsection{Poisson brackets} As in the previous section,  $\sfV = k^n$ and $x_0,\ldots,x_n$ are the standard linear coordinates on $\sfV$, so that the symmetric algebra is the polynomial ring $\mathsf{S}(\mathsf{V}^*) = \bk[x_0,...,x_{n-1}]$.

Given an additively alternating matrix $\ps$, we define a Poisson bracket on $\mathsf{S}(\sfV^*)$ by the formula
$$
\{x_i,x_j\}_\ps = \ps_{ij} x_i x_j,
$$
or equivalently as the bivector
\[
\sigma_\ps = \sum_{0\le i < j \le n-1} \ps_{ij}\partial_i \wedge \partial_j.
\]
We denote by $\mathsf{B}_\ps = (\mathsf{S}(\sfV^*),\{-,-\}_\ps)$ the resulting Poisson algebra. Note that $\sigma_\ps$ is invariant under the natural action of the torus $(\bkx)^n$, so the latter acts by automorphisms of $\mathsf{B}_\ps$.

Recall that the \emph{Poisson cohomology $\mathsf{HP}(\mathsf{B}_\ps)$} is the cohomology of the complex of alternating polyderivations
$$
\xymatrix{
\mathsf{S}(\sfV^*) \ar[r]^-{\mathsf{d}_\ps} & \mathsf{S}(\sfV^*) \otimes \mathsf{V} \ar[r]^-{\mathsf{d}_\ps} & \mathsf{S}(\sfV^*) \otimes \wedge^2\mathsf{V} \ar[r]^-{\mathsf{d}_\ps}&\mathsf{S}(\sfV^*) \otimes \wedge^3\mathsf{V} \ar[r]^-{\mathsf{d}_\ps}& ...
}
$$
where $\mathsf{d}_\ps$ is the Lichnerowicz differential, defined as the Schouten bracket with the bivector $\sigma_\ps$.  

As with the Hochschild cohomology above, the torus action breaks this complex into weight spaces $\mathsf{HP}(\mathsf{B}_\ps) = \sum_{\mathbf{w}} \mathsf{HP}(\mathsf{B}_\ps)^\mathbf{w}$ which have an explicit description as a function of the additively alternating matrix $\ps$.  Namely, the differential $\mathsf{d}_\ps$ acts on a polyvector $\rho$ of torus weight $\mathbf{w}=(w_0, ..., w_{n-1})$ by the formula
$$
\mathsf{d}_\ps \big(\rho\big) = \left(\sum_{0\le i < j \le n-1}
\ps_{ij} w_j x_i \partial_i  \right) \wedge \rho.
$$
Hence the relevant condition is the following additive counterpart of the Hochschild-contributing condition from \autoref{def:HHcontributing}.
\begin{definition}\label{def:HPcontributing}
Let $\ps$ be an additively alternating matrix.  A torus weight $\mathbf{w}=(w_0,...,w_{n-1})\in\mathbb{Z}^{n}$ is \textit{$\ps$-Poisson-contributing}, if the following conditions hold
\begin{enumerate}
\item The inequality $w_i\ge -1$ holds for every $i$, and
\item The identity
\begin{align}
\sum_{j=0}^{n-1} \ps_{ij} w_j = 0 \label{eq:HPcontrib}
\end{align}
holds for every $i$ such that $w_i\ge 0$.
\end{enumerate}
\end{definition}

We then have the following Poisson counterpart of \autoref{lm:HHtoric}.
\begin{lemma}\label{lm:HPtoric}
The weight component $\mathsf{HP}(\mathsf{B}_\ps)^\mathbf{w}$ is non-zero if and only if $\mathbf{w}$ is Poisson-contributing, in which case
$$
\mathsf{HP}(\mathsf{B}_\ps)^\mathbf{w} \cong \bk\rho \otimes \bigwedge \bigoplus\limits_{j:w_j\ge0} \bk x_j\partial_j,
$$
where $\rho$ is the unique, up to a scalar, polyvector of weight $\mathbf{w}$ of the minimal degree.
\end{lemma}

\subsection{Unimodularity}
\begin{definition}
An additively alternating matrix is \emph{normalized} if its rows sum to zero:
\[
\sum_j \ps_{ij} = 0
\]
for all $i$.  
\end{definition}

The following is the semiclassical analogue of the Calabi-Yau condition. A Poisson structure $\sigma$ on $\mathbb{C}^n$ is called \textit{unimodular} if there exists a holomorphic volume form on $\mathbb{C}^n$ invariant under all Hamiltonian vectors fields of $\sigma$. The following fact is well known.

\begin{lemma}\label{lm:unimodularEqualsNormalized}
Let $\lambda$ be an additively alternating matrix.  Then the Poisson algebra $\mathsf{B}_\lambda$ is unimodular if and only if $\lambda$ is normalized.
\end{lemma}

\section{Comparing Hochschild and Poisson cohomology}

\subsection{Relevant weights}
In what follows, we will be interested primarily in the piece of the Hochschild or Poisson cohomology that is homogeneous of weight zero with respect to the diagonal action of $\bkx \subset (\bkx)^n$, or equivalently the part corresponding to weights $\mathbf{w} = (w_0,\ldots,w_{n-1})$ such that $\sum_i w_i=0$.  Considering that $w_i \ge -1$ for any Hochschild or Poisson-contributing weight, we make the following definition.
\begin{definition}
A weight $\mathbf{w} \in \bZ^{n}$ is called \emph{relevant} if $w_i \geq -1$ for all $i$ and $\sum_i w_i = 0$.  
\end{definition}

\begin{lemma}
  There are finitely many relevant weights. Every weight of $\mathsf{HH}(\mathsf{A}_q)^{\bk^\times}$ and of $\mathsf{HP}(\mathsf{B}_\ps)^{\bk^\times}$ is relevant.
  \end{lemma}
  \begin{proof} We only need to prove the first assertion.  But for any choice of $m$ components $i$ where $w_i=-1$, the remaining components are nonnegative integers summing to $m$, of which there are finitely many choices.
\end{proof}

\subsection{The entrywise exponential} \label{sec:ee}
As noted above, the conditions \eqref{eq:HHcontrib} and \eqref{eq:HPcontrib} for a weight to be Hochschild-contributing and Poisson-contributing, respectively,  are similar: \eqref{eq:HHcontrib} follows from \eqref{eq:HPcontrib} by taking $q_{ij} = e^{\ps_{ij}}$.  To make this rigorous, we need to work in a setting where the exponential map is defined.  Thus, we now take $\bk =\bC$, though of course other options are possible.

\begin{definition}
The \emph{entrywise exponential} is the holomorphic map
\[
\mapdef{\EE}{\askew{\bC}}{ \mskew{\bC}}
{(\ps_{ij})_{i,j}}{ (\exp(\ps_{ij}))_{i,j} }
\]
\end{definition}

 The following is then immediate from the definitions:
\begin{lemma}
Let $\ps \in \askew{\bC}$ be an additively alternating matrix, and let $q = \EE(\ps) \in \mskew{\bC}$ be its entrywise exponential.  Then
\begin{itemize}
\item If $\ps$ is normalized, so is $q$.
\item If a weight $\mathbf{w} \in \bZ^n$ is $\ps$-Poisson contributing, then it is also $q$-Hochschild contributing.
\end{itemize}
\end{lemma}

\subsection{Genericity}

In general, it can happen that there are more $q$-Hochschild-contributing weights than $\ps$-Poisson-contributing ones, even if we restrict our attention only to the relevant weights.  To avoid this, it suffices to require that  $\sum_j \ps_{ij} w_j \notin 2 \pi \sqrt{-1} \bZ\setminus\{0\} \subset \bC$ for every (relevant) weight  $\mathbf{w}$. It follows that,  for every alternating matrix $\ps$, 
$\EE(\hbar \ps)$-Hochschild-contributing implies $\hbar \ps$-Poisson-contributing for all but countably many values of $\hbar$. 

There is a second natural open condition on $\ps$ to consider, which has a geometric motivation.  Namely, the Poisson bivector $\sigma_\ps$  on $\sfV = \bk^n$ descends to a toric Poisson bivector on the projective space $\bP(\sfV) = \bP^{n-1}$.  This bivector on projective space is symplectic away from the coordinate hyperplanes (hence log symplectic) if and only if the matrix $\ps$ has rank $n-1$.

Combining these two natural conditions, we arrive at the following.
\begin{definition}\label{def:ageneric}
An additively alternating matrix $\ps \in \askew{\bC}$ is \emph{generic} if the following conditions hold:
\begin{enumerate}
\item\label{it:nondegen} $\ps$ has corank 1, i.e.~rank $n-1$.
\item\label{it:contrib} Every relevant $q$-Hochschild-contributing weight is also $\ps$-Poisson-contributing, where $q = \EE(\ps)$.
\end{enumerate}
\end{definition}

\begin{definition}\label{def:mgeneric}
A multiplicatively alternating matrix $q \in \mskew{\bC}$ is \emph{generic} if $q = \EE(\ps)$ for some generic additively alternating matrix $\ps \in \askew{\bC}$. 
\end{definition}

The locus of generic multiplicatively alternating matrices $q$ has the following structure.
\begin{theorem}\label{thm:good-q}
  The set $X \subseteq \mskew{\bC}$ of generic multiplicatively alternating matrices is Zariski open.  Moreover, we have a decomposition into Zariski locally closed subsets
  \[
  X = \bigsqcup_{S} X_S
  \]  
  indexed by sets $S$ of relevant weights that are Poisson-contributing for some generic additively alternating matrix, and   
  \[
    X_S := \{q \in \mskew{\bC} \mid S \text{ is the set of relevant }q\text{-Hochschild-contributing weights\}},
  \]
  A similar statement holds for the subset $X' \subset X$ of normalized matrices.
\end{theorem}
This is a very explicit description, because each $X_S$ is just an intersection of hypersurfaces of the form $q^{\mathbf{w}}=1$, take away a union of other hypersurfaces of the same form. Moreover, there are only finitely many $S$ that can occur, since the set of relevant weights is finite. Remarkably, we do not need to impose any nondegeneracy condition on $q$ itself, merely a restriction on the set of contributing weights.
\begin{proof}[Proof of Theorem \ref{thm:good-q}]
  It is obvious that $X$ is contained in the union of the given sets $X_S$.  For the converse, given a set $S$ of relevant weights, let $Y_S \subseteq \askew{\bC}$ be the set of additively alternating matrices $\ps$ whose set of relevant $\ps$-Poisson-contributing weights is exactly $S$.
  
   Then we have a surjective  map $\varphi: \overline{Y_S} \to \overline{X_S}, \ps \mapsto \EE(\ps)$. Since every Poisson-contributing weight is Hochschild-contributing, this restricts to a map $\varphi: Y_S \to \overline{X_S}$ whose image contains $X_S$.  We just need to show that every point of $X_S$ is furthermore in the image of a matrix of rank $n-1$ lying in $Y_S$.  Let $Z \subseteq Y_S$ be the Zariski-closed locus of matrices of rank $<n-1$.  We wish to show that $X_S \subseteq \varphi(Y_S \setminus Z)$.  By assumption, $Y_S \setminus Z$ is nonempty. By definition, $Y_S$ is a linear subspace cut out by $\bZ$-linear equations. Therefore, $Y_S \cong \Lambda \otimes_\bZ \bC$ for some lattice $\Lambda$ in the set of alternating integer $n \times n$ matrices.  Now, $\varphi$ is unaffected by translation by $2 \pi \sqrt{-1} \Lambda$.  Since the latter is Zariski dense in $Y_S$, we get that $\bigcap_{z \in  \Lambda} (Z + 2 \pi \sqrt{-1} z) = \emptyset$.  As a result, $\varphi(Y_S \setminus Z) = \varphi(Y_S)$, which contains $X_S$.

  Finally, we show that $X$ is Zariski open.  Since it is a union of Zariski locally closed sets, it suffices to check that it is analytically open.  For this, now that given a generic element $\ps \in \askew{\bC}$ with $q=\EE(\ps)$, the map $\varphi$ is analytically locally an isomorphism from alternating matrices to the set $\mskew{\bC}$, so that deformations of $q$ in $\mskew{\bC}$ can automatically be obtained by deforming $\ps \in \askew{\bC}$.  Since the condition that $\ps$ has rank $n-1$ is open, so too is the condition of genericity of $q$. 
\end{proof}

\subsection{Biresidues and smoothing diagrams} \label{subsec:smoothingDiagrams}

Let $\ps$ be an additively alternating matrix. If $\ps$ is normalized, then the image of $\ps$ lies in the subspace
\[
\Delta := \set{ (z_0,...,z_{n-1})\in \bC^n}{z_0+...+z_{n-1}=0}
\]
Hence $\ps$ restricts to a linear map $\Delta \to \Delta$.  The following is straightforward.

\begin{lemma}\label{lm:biresidueMatrixDef}
If $\ps$ has rank $n-1$, then there is a unique normalized additively alternating matrix $b$ such that $b|_\Delta = \ps|_\Delta^{-1}$.  Moreover, this matrix is uniquely determined by the condition that the submatrix $(b_{ij})_{i,j=1}^{n-1}$ is the inverse of $(\widetilde{\ps}_{ij})_{i,j=1}^{n-1}$, where $\widetilde{\ps}_{ij} = \ps_{ij} + \ps_{j0} + \ps_{0i}$. 
\end{lemma}

The matrix $b$ in \autoref{lm:biresidueMatrixDef} is called the \textit{biresidue matrix} of $\ps$.  Geometrically, $b$ encodes the periods of the inverse log symplectic form on $\mathbb{P}^{n-1}$. It turns out that some aspects of the Poisson/Hochschild cohomology are easier to study via $b$, rather than $\ps$. The following proposition illustrates this:

\begin{proposition}\label{p:HHtoricLogSympl}
Let $\ps$ be a normalized alternating matrix of rank $n-1$.
Then a relevant weight is $\ps$-Poisson-contributing 
 if and only if $\mathbf{w}$ can be expressed as a linear combination of rows $b_{i\bullet}$ for which $w_i=-1$.  If, in addition, $\ps$ is generic, this is equivalent to $\mathbf{w}$ being $q$-Hochschild-contributing, for $q = \EE(\ps)$.
\end{proposition}

\begin{proof}
By \autoref{lm:HPtoric}, $\mathsf{HP}(\mathsf{B}_\ps)^\mathbf{w}\not=0$  implies that $w_i\ge-1$ for each $i$, and 
$\sum_{j=0}^{n-1}\ps_{ij}w_j=0, \text{~~whenever~~} w_i\ge0$.
Without loss of generality, let us assume that $w_0\ge0$. Then we have

\begin{equation}\label{eq:wtcontributesHP}
\sum_{j=1}^{n-1}\widetilde{\ps}_{ij}w_j =\sum_{j=1}^{n-1}(\ps_{ij}+\ps_{j0}+\ps_{0i})w_j =0, \text{~~whenever~~} w_i\ge0.
\end{equation}

The condition \eqref{eq:wtcontributesHP} can be rephrased as $\widetilde{\mathbf{w}} \widetilde{\ps}    = \sum_{i:w_i=-1}\alpha_i\mathbf{e}_i$ for some $\alpha_i \in \bC$, where $\mathbf{e}_i$ denotes the standard basis vector, $\widetilde{\mathbf{w}}=(w_1,...,w_{n-1})$ is a truncation of $\mathbf{w}$, and $\widetilde{\ps} = (\widetilde{\ps}_{ij})_{i,j=1}^{n-1}$ is the matrix from \autoref{lm:biresidueMatrixDef}.
Multiplying this equality by $(b_{ij})_{i,j=1}^{n-1}$ on the right, we obtain $\widetilde{\mathbf{w}}=\sum_{i:w_i=-1} \alpha_i \widetilde{b}_{i\bullet}$, where $\widetilde{b}_{i\bullet}=(b_{i1},...,b_{i,n-1})$. It remains to note that $\mathbf{w}$ and each row of $b$ have zero sum. Therefore, we have $\mathbf{w}=\sum_{i:w_i=-1} \alpha_i {b}_{i\bullet}$. This proves necessity. Sufficiency is proved in a similar way, by reversing all the implications above.
\end{proof}

In low degrees, we have the following possibilities for the contributing weights:
\begin{itemize}
\item $\HP^0(\sfB_\ps)^{\bCx}$ (the center): the only contributing weight is the zero weight $\mathbf{0}$.
\item $\HP^1(\sfB_\ps)^{\bCx}$ (infinitesimal symmetries):  again, the only contributing weight is the zero weight $\mathbf{0}$. Indeed, a weight with exactly one $-1$ at the $i$-th spot cannot be a multiple of $b_{i\bullet}$, because $b_{ii}=0$.
\item $\HP^2(\sfB_\ps)^{\bCx}$ (infinitesimal deformations): the contributing weights are of two types -- the zero weight $\mathbf{0}$ and contributing weights with exactly two $-1$. We are going to call the latter weights smoothable. Concretely, a weight $\mathbf{w}$ is \textit{smoothable} if there are two indices $i,j$ such that $w_i=w_j=-1$, $b_{ij}\not=0$, and 
\begin{equation}
\frac{b_{jk}+b_{ki}}{b_{ij}} = w_k \in\mathbb{Z}_{\ge0}, \text{~~~for } k\not=i,j.
\end{equation}
The terminology is justified by the fact the the corresponding first order deformation of the log symplectic form on $\mathbb{P}^{n-1}$ is smoothing out the singularity $\{x_i=x_j=0\}$ of the polar divisor.
\item $\HP^3(\sfB_\ps)^{\bCx}$ (obstructions to deformations): the contributing weights are of three types, namely the zero weight $\mathbf{0}$, the smoothable weights and the contributing weights with exactly three $-1$s. The last type of weights are called obstructed. Concretely, a weight $\mathbf{w}$ is \textit{obstructed} if there are three indices $i,j,k$ such that $w_i=w_j=w_k=-1$, and
$w_\ell\ge0$ for $\ell\not=i,j,k$, and $\mathbf{w}$ is a linear combination of the rows $b_{i\bullet}, b_{j\bullet}, b_{k\bullet}$.  
\end{itemize}

Following \cite{Matviichuk2020},  we are going to encode the information about $\mathsf{HH}^2(\mathsf{A}_q)$ using a \textit{smoothing diagram}, which is defined as a graph with some edges and angles colored according to the following recipe. In a complete graph on $n$ vertices, we color an edge $i\edge j$, and call it smoothable, if there is a smoothable weight $\boldsymbol{\theta}$ with $\theta_i=\theta_j=-1$. Note that, for each $i \neq j$, there can be at most one such $\boldsymbol{\theta}$. Additionally, we color each angle $i\edge k\edge j$
such that $i\edge j$ is smoothable and the $k$-th entry of the corresponding (uniquely defined) smoothable weight is non-zero. Note that one can associate a smoothing diagram to any alternating matrix $b$: a smoothable edge means that a (unique) linear combination of rows $i$ and $j$ has $-1$ in both the $i$-th and $j$-th component, and nonnegative integers in all other components.
 In the case at hand, each row of $b$ sums to zero.
 In this case, for a smoothable edge $i\edge j$, all entries of the corresponding
 smoothable weight $\boldsymbol{\theta}$ other than the $i$-th and $j$-th entries are zero, except either one $2$, or two $1$'s. Having this in mind, whenever $\boldsymbol{\theta}$ is a smoothable weight with $\theta_i=\theta_j=-1$, we are going to use dark coloring for the angle $i\edge k\edge j$, if $\theta_{k}=2$, and light coloring if $\theta_k=1$; see \autoref{fig:colored_angles_types}.

\begin{figure}[h]
\centering
\begin{subfigure}[t]{0.35\textwidth}\hspace{0.5cm}
\includegraphics[scale=1.0]{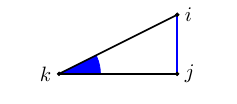}
\caption{Darkly colored angle, $\theta_k=2$}
\end{subfigure}
\hskip1cm
\begin{subfigure}[t]{0.35\textwidth}\hspace{0.5cm}
\includegraphics[scale=1.0]{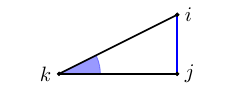}
\caption{Lightly colored angle, $\theta_k=1$}
\end{subfigure}

\caption{Two types of colored angles}
\label{fig:colored_angles_types}
\end{figure}

\begin{example}\label{ex:running1}
The following matrix $b$ has the smoothing diagram drawn on the right hand side:
$$
b= 
\begin{pmatrix}
0 &2& -4& -4& 6\\
-2& 0& 3& 1& -2 \\
4& -3& 0& 1& -2 \\
4& -1& -1& 0& -2 \\
-6& 2& 2& 2& 0
\end{pmatrix} \hskip0.2\textwidth
\raisebox{-1.5cm}{\includegraphics[scale=1.0]{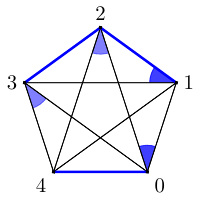}}
$$
Consequently, the corresponding log symplectic Poisson structure on $\bP^{4}$ admits a deformation whose degeneracy divisor is obtained by smoothing the intersections of the pairs of hyperplanes $(x_1,x_2)$, $(x_2,x_3)$ and $(x_0,x_4)$, resulting in a irreducible cubic hypersurface (deforming $x_1x_2x_3=0$) and an irreducible quadric hypersurface (deforming $x_0x_4=0$).
\end{example}

The following basic combinatorial constraint on smooth diagrams will play an important role later in the paper.
\begin{lemma}[\cite{Matviichuk2020}]\label{lm:pos_rat_valency}
Let $b=(b_{ij})_{i,j=0}^{n-1}$ be an alternating matrix. Then its smoothing diagram has the following properties.
\begin{enumerate}
\item For each vertex $j$, there are at most two smoothable edges containing $j$.  Hence the smoothing diagram is a disjoint union of chains and cycles of edges.
\item If $i\edge j$ and $j\edge k$ are two distinct smoothable edges, then $b_{ij}$ is a positive rational multiple of $b_{jk}$.
\end{enumerate}
\end{lemma}

\begin{proof}
Let us prove the second statement first. Suppose that
$$
\dfrac{b_{jk}+b_{ki}}{b_{ij}}=\alpha\in\mathbb{Z}_{\ge0} ~~~~\text{and} ~~~~\dfrac{b_{ki}+b_{ij}}{b_{jk}}=\beta\in\mathbb{Z}_{\ge0},
$$
By subtracting two equalities $b_{jk}+b_{ki} = \alpha~b_{ij}$ and $b_{ki}+b_{ij} = \beta~b_{jk}$, we obtain $(1+\beta) ~b_{jk} = (1+\alpha) ~b_{ij}$.

To prove the first statement, let us assume that $j\edge i$, $j\edge \ell$, and $j\edge k$ are pairwise distinct smoothable edges. Then, by the second statement, $b_{ji}$ and $b_{j\ell}$ are negative multiples of each other, and so are $b_{j\ell}$ and $b_{jk}$. However, those two facts together imply that $b_{ji}$ and $b_{jk}$ are positive multiples of each other, which contradicts the second statement.
\end{proof}

\section{Algebra deformations from smoothing diagrams}

Throughout this section, $\ps$ is a generic normalized alternating matrix, and $q = \EE(\ps)$.  Thus the Poisson and Hochschild cohomologies agree: $\HH(\sfA_q)^{\bCx} \cong \HP(\sfB_\ps)^{\bCx}$.  Consequently, the deformations of $\sfA_q$ can also be described in terms of smoothing diagrams, mirroring the corresponding Poisson deformations.  We now give a detailed description of this procedure, starting by explaining the contribution of each smoothable edge, then combining them into chains and cycles.

\subsection{Deformations from smoothable edges}

Let $i\edge j$, $i<j$, be a smoothable edge, with colored angles $i\edge k \edge j$ and $i\edge \ell \edge j$, where $k\le \ell$.  Let $R= \bC[\varepsilon]$, and define an $R$-algebra $\widetilde \sfA$ as the quotient of the free algebra $R\abrac{x_0,\ldots,x_{n-1}}$ by the relations 
$$
x_b x_a = \begin{cases}
q_{ji} x_i x_j +\varepsilon x_k x_\ell & a =i \textrm{ and }b= j \\
q_{ba} x_a x_b & \textrm{otherwise}
\end{cases}
$$
for all $b > a$, i.e.~we deform the relation $x_jx_i = q_{ji}x_ix_j$ to $x_jx_i = q_{ji}x_ix_j + \varepsilon x_kx_l$, leaving the other relations alone.    Clearly $\widetilde\sfA/\varepsilon\widetilde \sfA \cong \sfA_q$, and moreover one can check that the monomials $x_{i_1}...x_{i_p}$, $i_1\le ... \le i_p$ form an $R$-basis for $\widetilde \sfA$.  (This can be deduced from Bergman's Diamond Lemma \cite{Bergman1978}. It also follows from \autoref{prop:deformedKoszulNoCycle} below, so we omit the proof.) Thus $\widetilde \sfA$ is a flat $R$-deformation of $\sfA_q$.

\begin{example}\label{ex:running2}
Let $\mathsf{A}_q$ be the non-commutative algebra corresponding to the biresidue matrix $b$ from \autoref{ex:running1}. The matrix $\ps=(\ps_{ij})_{i,j=0}^{n-1}$ of the Poisson bracket coefficients is given by
$$
\ps = \dfrac{1}{30} 
\begin{pmatrix}
0 & 1 & -1 & 3 & -3 \\
-1 & 0 & -12 & 12 & 1 \\
1 & 12 & 0 & -6 & -7 \\
-3 & -12& 6 & 0 & 9 \\
3 & -1 & 7 & -9 & 0
\end{pmatrix}.
$$
Let $q = \EE(\ps)$, and let $v = e^{\frac{1}{30}}$. We then have the following relations for $\mathsf{A}_q$:
\begin{align*}
x_1 x_0 &= v^{-1} x_0 x_1 &  x_2 x_0 &= v x_0 x_2 & x_3 x_0 &= v^{-3} x_0 x_3 & x_4 x_0 &= v^3 x_0 x_4 \\
x_2 x_1 &= v^{12} x_1x_2 & x_3 x_1 &= v^{-12} x_1 x_3 & x_4x_1 &= v^{-1} x_1 x_4 \\
x_3 x_2 &= v^6 x_2 x_3 & x_4 x_2 &= v^7 x_2 x_4 \\ x_4 x_3 &= v^{-9}x_3x_4.
\end{align*}

The smoothing diagram for $b$ has three smoothable edges: $1\edge 2$, $2\edge 3$ and $0\edge 4$. The corresponding smoothable weights are $\boldsymbol{\theta}_{12} = (2,-1,-1,0,0)$, $\boldsymbol{\theta}_{23} = (0,2,-1,-1,0)$ and $\boldsymbol{\theta}_{04}=(-1,0,1,1,-1)$. These give three different ways to deform $\mathsf{A}_q$. For the first we deform the relation $x_2x_1=v^{12} x_1 x_2$ to $x_2x_1=v^{12} x_1 x_2 + \varepsilon x_0^2$. For the second, we deform $x_3x_2 = v^6 x_2 x_3$ to $x_3 x_2 = v^6 x_2 x_3 + \varepsilon x_1^2$. For the third, we deform $x_4 x_0 = v^3 x_0 x_4$ to $x_4 x_0 = v^3 x_0 x_4 + \varepsilon  x_2x_3$.
\end{example}

Sometimes the deformations coming from different smoothable edges can be combined in an obvious way. For instance, in \autoref{ex:running2} the first two deformations can be combined by deforming both relations $x_2x_1=v^{12} x_1 x_2$ and $x_3x_2 = v^6 x_2 x_3$ as described above. One can check that the obtained deformation of $\mathsf{A}_q$ still has the PBW basis $x_{i_1}...x_{i_p}$, $i_1\le ... \le i_p$. However, if we try to combine, say, the last two deformations in the same obvious way, we run into trouble: the two following ways of reducing the word $x_3 x_4 x_0$ produce different results.  On the one hand,
$$
x_4 x_3 x_0  =v^{-9} x_3 x_4 x_0 = v^{-9} x_3 (v^3 x_0 x_4 + \varepsilon x_2 x_3) = v^{-9} x_0 x_3 x_4 + \varepsilon v^{-3} x_2 x_3^2 + \varepsilon^2 v^{-9}x_1^2 x_3,
$$
and on the other hand,
$$
x_4x_3 x_0 = x_4 v^{-3} x_0 x_3 = (v^3 x_0 x_4 + \varepsilon x_2 x_3) v^{-3} x_3 = v^{-9} x_0 x_3 x_4 + \varepsilon v^{-3} x_2 x_3^2.
$$
Hence the resulting deformation is not flat.  We now explain how to fix this issue.

\subsection{Cycle-free deformations} \label{sec:nocycles}

In this subsection, we fix a biresidue matrix $b = (b_{ij})^{n-1}_{i,j=0}$, and a collection of smoothable edges that does not contain any cycles.
Then by permuting the indices if necessary, we may assume that the edges in the collection have the form
$i\edge (i+1)$, $i\in I$ for some indexing set $I \subset \{0,\ldots,n-2\}$.

We will prove that the deformations corresponding to all of the edges in this collection can be combined in an essentially unique way, by showing that the obstructions to such deformations vanish by torus weight considerations.  To do so, we need some linear algebraic statements that constrain the possible torus weights.  First we treat the weights that occur in the possible deformations:

\begin{lemma}\label{lm:smoothableWeightLinIndep}
For a subset $I \subset \{0,\ldots,n-2\}$ indexing a cycle-free collection of smoothable edges of $b$ as above, the smoothable weights $\boldsymbol{\theta}_i$, $i\in I$, are linearly independent over $\mathbb{C}$.
\end{lemma}

\begin{proof}
For $i\in I$, let $$\boldsymbol{\theta}_{i} = \frac{1}{b_{i,i+1}}\Big(-b_{i,\bullet}+b_{i+1,\bullet}\Big)$$ be the smoothable weight corresponding to $i\edge (i+1)$. Assuming a $\mathbb{C}$-linear combination $\sum_{i\in I} \nu_i \boldsymbol{\theta}_{i}$ is zero, we get $\sum_{k=0}^{n-1}\mu_k b_{k,\bullet} = \mathbf{0}$, where the constants $\mu_0, ..., \mu_{n-1}$ sum to zero. Since the alternating matrix $b$ has corank $1$, there is essentially only one non-trivial zero linear combination of its rows, namely $\sum_{k=0}^{n-1} b_{k,\bullet}=\mathbf{0}$. We therefore have $\mu_0=...=\mu_{n-1}$, which together with $\mu_0+...+\mu_{n-1}=0$ implies that all $\mu$s are zero. By induction on $i$, this implies that each $\nu_i$ is zero.
\end{proof}

Next we treat the weights of the possible obstructions. The proof of the following statement is elementary but long, so we delay it to the end of this subsection.
\begin{lemma}\label{lm:avoidingObstructedWeights}
For a subset $I \subset \{0,\ldots,n-2\}$ indexing a cycle-free collection of smoothable edges of $b$ as above, no non-negative  integer linear combination of the smoothable weights $\boldsymbol{\theta}_i$, $i\in I$ is obstructed.
\end{lemma}

We now use the chosen set of smoothable edges $i\edge(i+1)$, $i\in I$ to define a filtration by ideals
\[
\mathsf{T}(\mathsf{V}^*) = \mathcal{F}_0 \supseteq \mathcal{F}_1 \supseteq \mathcal{F}_2 \supseteq ...
\]
as follows. For $m \ge 0$, let 
\[
\Theta_m = \{\boldsymbol{\theta}_{i_1}+...+\boldsymbol{\theta}_{i_m}: i_1,...,i_m \in I\}
\] 
be the set of weight vectors that can be written as a sum of $m$ of the chosen smoothable weights.  In particular $\Theta_0 = 0$, and $\Theta_1$ is the set of chosen smoothable weights.  Denoting the monomial of weight ${\bf w}$ by ${\bf x}^{\bf w} = x_0^{w_0} \dots x_{n-1}^{w_{n-1}}$, we define the filtration by declaring that
\begin{align}\label{eq:smoothableFiltration}
{\bf x}^{\bf w}  \in \mathcal{F}_a \textrm{ if and only if } {\bf w}\in \Theta_m+\mathbb{Z}_{\ge0}^n \textrm { for some } m\ge a.
\end{align}

Note that $\mathcal{F}_a \mathcal{F}_b \subseteq \mathcal{F}_{a+b}$, and that the linear independence of $\Theta_1$ (\autoref{lm:smoothableWeightLinIndep}) guarantees that $\min\limits_{0 \neq f\in\mathcal{F}_a} {\rm deg}(f) \to \infty$ as $a\to \infty$. In particular, we have $\cap_{a\ge 0} \mathcal{F}_a = 0$.

\begin{theorem}\label{thm:quantumTorelliNoCycles}
For every choice of non-zero constants $\gamma_i\in\mathbb{C}$, $i\in I$, the algebra $\mathsf{A}_q$ admits a flat filtered deformation to a quadratic $\mathbb{C}[\varepsilon]$-algebra $\mathsf{T}(\mathsf{V}^*)[\varepsilon]/J_{q,I}$,  whose two-sided ideal $J_{q,I}$ of relations is generated by quadratic elements of the form 
\begin{equation}\label{eq:deformedRelationsNoCycles}
r_{ab}(\varepsilon) = - x_b x_a + q_{ba} x_a x_b + \sum_{m=1}^N \varepsilon^m \left(\sum_{\substack{-\mathrm{e}_a-\mathrm{e}_b+\mathrm{e}_k+\mathrm{e}_\ell \in \Theta_m\\k\le \ell}} C_{ba}^{k\ell} x_k x_\ell\right),~~~a<b,
\end{equation}
for some $N \ge 1$, 
where $C_{ba}^{k\ell}=\gamma_i$ whenever $-\mathrm{e}_a-\mathrm{e}_b+\mathrm{e}_k+\mathrm{e}_\ell =\boldsymbol{\theta}_i$. Specializing to $\varepsilon\in\mathbb{C}^\times$, we obtain a quadratic $\mathbb{C}$-algebra $\mathsf{A}_{q,I}$ that is independent, up to a filtered automorphism of $\mathsf{V}^*$, of the non-zero constants $\gamma_i$ or $\varepsilon$.

Moreover, if $\Theta_m$ does not contain a weight of the form $e_a - e_b$ for any $m\geq1$, then the constants $C_{ba}^{kl}$ in \eqref{eq:deformedRelationsNoCycles} are uniquely determined by $\gamma_i,i\in I$.
\end{theorem}

\begin{proof}
This is a standard deformation-obstruction argument using the equivalence between flat deformations of $\sfA_q$ and Maurer--Cartan elements in the Hochschild cochain complex.

For $i\in I$, let $\mu^{(i)}\in {\rm Hom}_\mathbb{C}(\mathsf{A}_q\otimes \mathsf{A}_q,\mathsf{A}_q)$ be a Hochschild $2$-cocycle representing the unique, up to scalar, Hochschild cohomology class in $\mathsf{HH}^2(\mathsf{A}_q)^{\boldsymbol{\theta}_i}$. 
We are going to use an inductive argument to construct an associative product of the form
$$
\star_\varepsilon = \sum_{m=0}^\infty \varepsilon^m \mu_{m},
$$
whose constant term $\mu_{0}$ is the original product on $\sfA_q$, whose linear term $\mu_1= \sum_{i\in I} \gamma_i \mu^{(i)}$ is the specified first-order deformation, and whose higher order terms have the form $\mu_m \in {\rm Hom}_{\mathbb{C}}(\mathsf{A}_q\otimes\mathsf{A}_q,\mathsf{A}_q)^{{\Theta}_m}$, $m\ge 2$. For this suppose that we have managed to find the $2$-cocycles $\mu_2,..., \mu_{m-1}$ such that the Gerstenhaber bracket $[\sum_{k=0}^{m-1}\varepsilon^k\mu_k,\sum_{k=0}^{m-1}\varepsilon^k\mu_{k}]_G$ is a multiple of $\varepsilon^{m}$. To find appropriate $\mu_{m}$, we look at the obstruction
$$
{\rm Obs}_m = -\frac{1}{2} \sum_{k=1}^{m-1} [\mu_{k},\mu_{m-k}]_G \in {\rm Hom}_\mathbb{C}(\mathsf{A}_q^{\otimes 3}, \mathsf{A}_q)^{{\Theta}_m}.
$$
The Jacobi identity for the Gerstenhaber bracket implies 
that ${\rm Obs}_m$ is a Hochschild $3$-cocycle, and hence \autoref{lm:avoidingObstructedWeights} and \autoref{lm:smoothableWeightLinIndep} imply that it is a Hochschild coboundary, i.e.

$$
{\rm Obs}_m = \mathsf{d}(\mu_{m}), 
$$
for a $2$-cochain $\mu_{m}\in {\rm Hom}_\mathbb{C}(\mathsf{A}_q\otimes \mathsf{A}_q,\mathsf{A}_q)^{{\Theta}_m}$. By adding the term $\varepsilon^m \mu_{m}$ to the power series expansion, we obtain that $[\sum_{k=0}^{m}\varepsilon^k\mu_k,\sum_{k=0}^{m}\varepsilon^k\mu_{k}]_G$ is a multiple of $\varepsilon^{m+1}$. The inductive process will eventually terminate, because $\mu_m(x_a,x_b)\in \mathcal{F}_{\alpha+\beta+m}$, for $x_a\in\mathcal{F}_\alpha$, $x_b\in\mathcal{F}_\beta$. This finishes the construction of $\star_{\varepsilon}$.

The variables $x_i$ generate the obtained algebra, since it is a filtered deformation of $\mathsf{A}_q$. Moreover, by working modulo $\varepsilon^m$, inductively in $m$, we arrive at the relations  \eqref{eq:deformedRelationsNoCycles}.

For the uniqueness of $\mathsf{A}_{q,I}$, we use the fact that $\mathsf{HH}^2(\mathsf{A}_q)^{\Theta_m}=0$, for $m\ge 2$, which follows from \autoref{lm:smoothableWeightLinIndep}. Any two choices of $\mu_m$ differ by $\mathsf{d}(\nu_m)$, for a Hochschild $1$-cochain $\nu_m\in{\rm Hom}_\mathbb{C}(\mathsf{A}_q,\mathsf{A}_q)^{\Theta_m}$. Therefore, the gauge equivalence $\exp(\varepsilon^m \nu_m)$ defines an isomorphism between the two choices of $\star_\varepsilon$, modulo $\varepsilon^{m+1}$. Restricting $\exp(\varepsilon^m \nu_m)$ to $\mathsf{V}^*$, we obtain the desired filtered automorphism intertwining the two choices of relations \eqref{eq:deformedRelationsNoCycles}, modulo $\varepsilon^{m+1}$.

By choosing a diagonal transformation $x_a \mapsto t_a x_a$, $a=0,1,...,n-1$, so that $(t_0, ..., t_{n-1})\cdot \boldsymbol{\theta}_i = \gamma_i$, $i\in I$, up to isomorphism, we can make all $\gamma_i=1$. By choosing a further diagonal transformation with $(t_0, ..., t_{n-1})\cdot \boldsymbol{\theta}_i = \varepsilon$, $i\in I$ we can also make $\varepsilon=1$.

Finally, suppose that $\Theta_m$, $m\ge 1$, do not contain weights of the form $e_a - e_b$.  Then in the proof of uniqueness of $\mathsf{A}_{q,I}$, the restriction $\exp(\varepsilon^m \nu_m)$ to $\mathsf{V}^*$ has to be identity. Therefore, the constants $C_{ba}^{k\ell}$ are uniquely determined, once $\gamma_i$ are fixed.
\end{proof}

Let us call the weights of the form $-\mathrm{e}_a-\mathrm{e}_b+\mathrm{e}_k+\mathrm{e}_\ell$, $a\not=b$, \textit{quadratic}. Note that these are the sort of weights appearing in the relations \eqref{eq:deformedRelationsNoCycles}. Also, each smoothable weight is automatically quadratic.

\begin{example}\label{ex:running3}
Continuing with \autoref{ex:running2}, let us find a deformation that involves all three smoothable edges $1\edge 2$, $2\edge 3$ and $0\edge 4$. Recall that the smoothable weights are $\boldsymbol{\theta}_{12} = (2,-1,-1,0,0)$, $\boldsymbol{\theta}_{23} = (0,2,-1,-1,0)$ and $\boldsymbol{\theta}_{04}=(-1,0,1,1,-1)$; they form the set $\Theta_1$. 
The quadratic weights in $\Theta_2$  are
$$
\boldsymbol{\theta}_{12} + \boldsymbol{\theta}_{04} = (1,-1,0,1,-1),
\text{~~~~~and~~~~~}
\boldsymbol{\theta}_{23} + \boldsymbol{\theta}_{04} = (-1,2,0,0,-1).
$$
The set $\Theta_3$ has only one quadratic weight: $$\boldsymbol{\theta}_{12} + \boldsymbol{\theta}_{23} + \boldsymbol{\theta}_{04} = (1,1,-1,0,-1).$$
The remaining $\Theta_m$, $m\ge 4$, do not have quadratic weights. Therefore, we should be looking at the deformation given by the relations
$$
\begin{matrix}
x_1 x_0 &=& v^{-1} ~x_0 x_1, \\
x_2 x_0 &=& v ~x_0 x_2, \\
x_3 x_0 &=& v^{-3}~ x_0 x_3, \\
x_4 x_0 &=& v^3 ~x_0 x_4& + ~\varepsilon~ x_2x_3&+~\varepsilon^2~ C_{40}^{11}~ x_1^2,\\
x_2 x_1 &=& v^{12}~ x_1x_2& +~\varepsilon~ x_0^2, \\
x_3 x_1 &=& v^{-12} ~x_1 x_3, \\ 
x_4x_1 &=& v^{-1} ~x_1 x_4 & &+~\varepsilon^2~C_{41}^{03}~x_0x_3, \\
x_3 x_2 &=& v^6~ x_2 x_3&+~\varepsilon~ x_1^2, \\
x_4 x_2 &=& v^7 ~x_2 x_4 & & & + ~\varepsilon^3~C_{42}^{01}~x_0x_1, \\
x_4 x_3 &=& v^{-9}~x_3x_4.
\end{matrix}
$$
To find the values of the coefficients $C_{ba}^{k\ell}$, we use the principle of \textit{confluence}. Specifically, we pick a cubic word $x_\gamma x_\beta x_\alpha$, $\gamma>\beta>\alpha$, and reduce it in two ways to a linear combination of cubic words, in which the variables appear in non-decreasing order. This procedure will impose equations on $C_{ba}^{k\ell}$, which we know must admit solutions. For instance, looking at the word $x_4x_3x_0$, we obtain the following two reductions:
\begin{align*}
x_4 x_3 x_0  &=v^{-9} x_3 x_4 x_0 \\
&= v^{-9} x_3 (v^3 x_0 x_4 + \varepsilon x_2 x_3 + \varepsilon^2 C_{40}^{11} x_1^2) \\
&= v^{-9} x_0 x_3 x_4 + \varepsilon v^{-3} x_2 x_3^2 + \varepsilon^2 v^{-9}(1+C_{40}^{11}v^{-24})x_1^2 x_3,
\end{align*}
and
\begin{align*}
x_4x_3 x_0 &= x_4 v^{-3} x_0 x_3 \\
&= (v^3 x_0 x_4 + \varepsilon x_2 x_3+ \varepsilon^2 C_{40}^{11} x_1^2) v^{-3} x_3 \\
&= v^{-9} x_0 x_3 x_4 + \varepsilon v^{-3} x_2 x_3^2 + \varepsilon^2 C_{40}^{11}v^{-3} x_1^2x_3.
\end{align*}
Equating the $\varepsilon^2$ terms, we deduce that $C_{40}^{11} = \frac{v^{24}}{v^{30}-1}$. Similarly, we find $C_{41}^{03} = \frac{v^{-10}}{1-v^{5}}$ and  $C_{42}^{01} = \frac{v^2(1+v^{10})}{(1-v^5)(1-v^{30})}$.
\end{example}

\begin{example}\label{ex:4chain}
In \cite{Matviichuk2020}, we obtained a classification of all possible smoothing diagrams for $n=5$ with the aid of a computer.  The following is in some sense the most complicated cycle-free smoothing diagram: the deformed quadratic relations have a term of order $\varepsilon^6$, which is the maximal possible when $n=5$ and the smoothing diagram has no cycles. 
$$
b = (b_{ij})_{i,j=0}^4=
\begin{pmatrix}
0& 3& -1& -1& -1\\
-3& 0& 2& 2& -1\\
1& -2& 0& 2& -1\\
1& -2& -2& 0& 3\\
1& 1& 1& -3& 0
\end{pmatrix}\hskip0.2\textwidth
\raisebox{-1.5cm}{\includegraphics[scale=1.0]{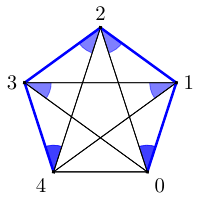}}
$$

The matrix $\ps=(\ps_{ij})_{i,j=0}^4$ of the Poisson bracket coefficients is given by
$$
\ps = \dfrac{1}{30} 
\begin{pmatrix}
0 & -6 & 6 & -2 & 2 \\
6 & 0 & -3 & -1 & -2 \\
-6 & 3 & 0 & -3 & 6 \\
2 & 1& 3 & 0 & -6 \\
-2 & 2 & -6 & 6 & 0
\end{pmatrix}.
$$

Setting $q_{ji} = e^{\ps_{ji}}$, we obtain the toric algebra $\mathsf{A}_q = \mathbb{C}\abrac{x_0,..,x_4} / (x_j x_i - q_{ji} x_i x_j)$. The  set $\Theta_1$ consists of the smoothable weights 
$$
\begin{matrix}
\boldsymbol{\theta}_0 &=& (&-1,&-1,&1,&1,&0&), \\
\boldsymbol{\theta}_1 &=& (&2,&-1,&-1,&0,&0&), \\
\boldsymbol{\theta}_2 &=& (&0,&0,&-1,&-1,&2&), \\
\boldsymbol{\theta}_3 &=& (&0,&1,&1,&-1,&-1&), \\
\end{matrix}
$$
and the quadratic weights appearing in $\Theta_m$, $m\ge2$, are as follows:
$$
\begin{matrix}
\Theta_2: & \boldsymbol{\theta}_0+\boldsymbol{\theta}_2 &=& (&-1,&-1,&0,&0,&2&), \\
& \boldsymbol{\theta}_0+\boldsymbol{\theta}_3&=&(&-1,&0,&2,&0,&-1&),\\
& \boldsymbol{\theta}_1+\boldsymbol{\theta}_3&=&(&2,&0,&0,&-1,&-1&),\\
\Theta_3: & \boldsymbol{\theta}_0+\boldsymbol{\theta}_2+\boldsymbol{\theta}_3&=&(&-1,&0,&1,&-1,&1&),\\
& \boldsymbol{\theta}_0+\boldsymbol{\theta}_1+\boldsymbol{\theta}_3&=&(&1,&-1,&1,&0,&-1&),\\
\Theta_4: & \boldsymbol{\theta}_0+\boldsymbol{\theta}_1+\boldsymbol{\theta}_2+\boldsymbol{\theta}_3&=&(&1,&-1,&0,&-1,&1&),\\
\Theta_5: & \emptyset \\
\Theta_6: & 2\boldsymbol{\theta}_0+\boldsymbol{\theta}_1+\boldsymbol{\theta}_2+2\boldsymbol{\theta}_3&=&(&0,&-1,&2,&-1,&0&),\\
\end{matrix}
$$

Therefore, there exists a deformation $\mathsf{A}_{q,I}$ with the quadratic relations the following form,
$$
\begin{matrix}
x_1x_0 &=& v^6~x_0x_1 &+\varepsilon~x_2x_3&+\varepsilon^2C_{10}^{44}~x_4^2,\\
x_2x_0 &=& v^{-6}~x_0x_2, \\
x_2x_1 &=& v^3~x_1x_2&+\varepsilon~x_0^2, \\
x_3x_0 &=& v^2~x_0x_3 &&&+\varepsilon^3C_{30}^{24}~x_2x_4,\\
x_3x_1 &=& v~x_1x_3 &&&&+\varepsilon^4C_{31}^{04}~x_0x_4&~~~~~&+\varepsilon^6C_{31}^{22}~x_2^2, \\
x_3x_2 &=& v^3~x_2x_3 & +\varepsilon~ x_4^2,\\
x_4x_0 &=& v^{-2}~x_0x_4 &&+\varepsilon^2C_{40}^{22}~x_2^2,\\
x_4x_1 &=& v^2~x_1x_4 &&&+\varepsilon^3C_{41}^{02}~x_0x_2,\\
x_4x_2 &=& v^{-6}~x_2x_4, \\
x_4x_3 &=& v^6~x_3x_4&+\varepsilon~x_1x_2 &+\varepsilon^2C_{43}^{00}~x_0^2.
\end{matrix}
$$
where $v = e^{1/30}$.

By solving the confluence equations for the constants $C_{ba}^{k\ell}$, starting from the $\varepsilon^2$ coefficients and moving inductively to the $\varepsilon^6$ coefficient, we obtain the constants
$$
\begin{matrix}
C_{10}^{44} & =& C_{43}^{00} & = & -\dfrac{v^{12}}{1-v^{15}}, &&C_{31}^{04} & = & \dfrac{v^5(1-v^5)(1+v^{10})}{(1-v^{10})^2(1-v^{15})^2},\\
&&C_{40}^{22} & = & \dfrac{1}{v^5(1-v^{10})}, \\
C_{30}^{24} & =&C_{41}^{02} & = & \dfrac{1+v^{10}}{v^5(1-v^{10})(1-v^{15})}, &&C_{31}^{22} &=& \dfrac{v^2(1+v^5+v^{15})}{(1-v^5)^3(1+v^5)^4(1-v^{15})^2}.\\
\end{matrix}
$$
which determine the algebra $\sfA_{q,I}$.
\end{example}

\begin{proof}[Proof of \autoref{lm:avoidingObstructedWeights}]  We must show that no non-negative integer linear combination of smoothable weights is obstructed.   Suppose to the contrary that ${\bf w} = \sum_{i\in I}\nu_i \boldsymbol{\theta}_i$ is obstructed, where each $\boldsymbol{\theta}_i$ is a smoothable weight and each $\nu_i\in \mathbb{Z}_{\ge0}$.    
  Then, for some indices $i_1 < i_2 < i_3$, we have $w_{i_1}=w_{i_2}=w_{i_3}=-1$ and $\mathbf{w} = -\mu_1 b_{i_1,\bullet} + \mu_2 b_{i_2,\bullet} + \mu_3 b_{i_3,\bullet}$.  We will use these equations to derive a contradiction.

  First, we claim that the following holds:
  \begin{gather}
    b_{i_1,i_2}+b_{i_2,i_3}+b_{i_3,i_1} = 0,\label{e:b-sum} \\
    -\mu_1 + \mu_2 + \mu_3 = 0.     \label{e:mu-sum}
  \end{gather}
  which immediately implies that 
    \begin{equation}\label{e:w-lin-comb}
   {\bf w} = \mu_1(-b_{i_1,\bullet}+b_{i_2,\bullet}) + \mu_3(-b_{i_2,\bullet}+b_{i_3,\bullet}).
    \end{equation}  
    
Indeed, for \eqref{e:b-sum}, consider the three-by-three matrix $B' := (b_{i_j,i_k})_{j,k=1}^3$, which is alternating and nonzero, and therefore has rank two.  The vector $(-1,-1,-1)$ lies in its row span, and hence without loss of generality, we can write this vector as a linear combination of the first two rows.  Hence there exist $c_1, c_2$ such that $c_1(0,b_{i_1,i_2},b_{i_1,i_3})+c_2(-b_{i_1,i_2},0,b_{i_2,i_3}) = (-1,-1,-1)$. This implies that $c_2=-c_1$, which in turn yields \eqref{e:b-sum}.  

 Meanwhile, for \eqref{e:mu-sum}, it suffices to show that the solutions of \eqref{e:mu-sum} are exactly the triples $(\mu_1,\mu_2,\mu_3)$  for which the vector
  $v := (-\mu_1 b_{i_1,i_j} + \mu_2 b_{i_2,i_j} + \mu_3 b_{i_3,i_j})_{j = 1}^3$ is a scalar multiple of $(-1,-1,-1)$.  Indeed, since the difference of any two of the rows of $B'$ is a multiple of $(-1,-1,-1)$,  all  $\mu_1,\mu_2,\mu_3$ satisfying \eqref{e:mu-sum} must yield $v \in \bC \cdot (-1,-1,-1)$.  Conversely, the condition on $\mu_1,\mu_2,\mu_3$ to have this is linear and nontrivial, as $b_{i_1,i_1}=0$.  As a result, \eqref{e:mu-sum} must be satisfied in order to have $v_1=v_2=v_3$.

Next, let us write each smoothable weight as
\begin{equation}\label{e:theta-i}
    \boldsymbol{\theta}_i = c_i (-b_{i,\bullet}+b_{i+1,\bullet}), \text{ for some } c_i \in \bC^*.
  \end{equation}
  It follows that 
  \[ \mathbf{w} = \sum_{i=i_1}^{i_2-1} c_i^{-1}\mu_1 \boldsymbol{\theta}_i + \sum_{i=i_2}^{i_3-1} c_i^{-1} \mu_3 \boldsymbol{\theta}_i. \]
  Linear independence of the $\boldsymbol{\theta}_i$ then implies that
  \begin{equation}\label{e:lambda-i}
    \nu_i = \begin{cases} c_i^{-1} \mu_1, & \text{if }  i_1 \leq i < i_2, \\
      c_i^{-1} \mu_3, & \text{if }  i_2 \leq i < i_3, \\
      0, & \text{otherwise.}
      \end{cases}
   \end{equation}
   We claim that $\mu_1, \mu_3 \neq 0$.
 Since $w_{i_1}=-1$, the fact that $\theta_{i,i_1} \geq 0$ for $i > i_1$  implies that $\nu_{i_1} > 0$. Thus $\mu_1 \neq 0$.  Similarly, $w_{i_3} = -1$ implies $\nu_{i_3} > 0$, so $\mu_3 \neq 0$.  As a consequence, $\nu_i > 0$ for all $i_1 \leq i < i_3$, with $\nu_i = 0$ otherwise. In particular, the edges $i \edge (i+1)$ are all smoothable for $i_1 \leq i < i_3$.

By \autoref{lm:pos_rat_valency}.2, smoothability of these edges implies that $b_{i,i+1}$, $i_1\le i <i_3$, are positive rational multiples of each other. Hence without loss of generality, we may assume that they are all positive. Smoothability then yields the following array of inequalities between real numbers:
\begin{equation}\label{e:inequalities}
\begin{matrix}
b_{i_1,i_1+1} & \ge & b_{i_1,i_1+2} & \ge & ... & \ge & \boxed{b_{i_1,i_2}} & \ge & b_{i_1,i_2+1} & \ge & ... & \ge & \boxed{b_{i_1,i_3}} \\
&& \rotatebox[origin=c]{90}{$\ge$} &&& & \rotatebox[origin=c]{90}{$\ge$}&& \rotatebox[origin=c]{90}{$\ge$} &&&& \rotatebox[origin=c]{90}{$\ge$} \\
&& b_{i_1+1,i_1+2} &\ge&...&\ge& b_{i_1+1,i_2}&\ge& b_{i_1+1,i_2+1} & \ge&...&\ge& b_{i_1+1,i_3} \\
&&&&&&&& \rotatebox[origin=c]{90}{$\ge$}  &&&& \rotatebox[origin=c]{90}{$\ge$} \\
&&&&& & &&\vdots&&&&\vdots \\
&&&&&&&& \rotatebox[origin=c]{90}{$\ge$}  &&&& \rotatebox[origin=c]{90}{$\ge$} \\
&&&&&&&&b_{i_2,i_2+1} & \ge & ... & \ge & \boxed{b_{i_2,i_3}} \\
&&&&&&&&   &&&& \vdots\\
&&&&&&&&   &&&& \rotatebox[origin=c]{90}{$\ge$} \\
&&&&&&&&&&&& b_{i_3-1,i_3} \\
\end{matrix}
\end{equation}
As a consequence of this, and the inequalities $\nu_i >0$ for $i_1 \leq i < i_3$, we get that  $c_i > 0$ for all $i_1 \leq i < i_3$. By \eqref{e:lambda-i}, we conclude that $\mu_1, \mu_3 > 0$.

We now derive a contradiction by showing that the positivity of $\mu_1$ and $\mu_3$ is incompatible with \eqref{e:w-lin-comb}.
The inequalities \eqref{e:inequalities} yield
\begin{equation}
b_{i_1 i_3} \leq \min(b_{i_1 i_2}, b_{i_2 i_3}).
\end{equation}
Thanks to  \eqref{e:b-sum} we get $b_{i_1 i_2} + b_{i_2 i_3} \leq \min(b_{i_1 i_2}, b_{i_2 i_3})$. Thus, $b_{i_1 i_2}, b_{i_2 i_3} \leq 0$.

However, \autoref{e:w-lin-comb} for $\bullet=i_2$ gives
\begin{align}
  \label{e:w-i2}
  \mu_1(-b_{i_1 i_2}) + \mu_3(b_{i_3 i_2}) &= -1.
\end{align}
But the left-hand side is  non-negative, a contradiction. \qedhere
\end{proof}

\subsection{Deforming along a full cycle}
\label{subsec:FOcycle}

We now consider deformations obtained from a smoothing diagram given by a cycle involving all of the vertices.  Such smoothing diagrams were classified by the first author in \cite{Matviichuk2023b}; let us recall the result.  

Because of the cyclic symmetry, it is convenient in this subsection to consider the indices of the variables $x_0,\ldots,x_{n-1}$ as elements in $\bZ/n\bZ$.  If $k$ is an integer, we denote by $\overline{k}$ its remainder modulo $n$, so that $0 \leq \overline{k}< n$. 
Given an integer $0 < k < n$ such that with $\gcd(n,k)=\gcd(n,k+1)=1$, we define a normalized matrix $\FO{n}{k} \in \askew{\bC}$ by declaring that its entries are given by the formula 
\begin{align}
\FO[ij]{n}{k} &:= \overline{j-i} + \overline{k(j-i)} - n & 0 \le i\neq j < n \label{eq:FOPoissonMatrix}
\end{align}

\begin{proposition}[{\cite[Proposition 4.2 and Corollary 3.4]{Matviichuk2023b}}]\label{prop:cycle-classification}
Let $\ps\in\askew{\bC}$ be a normalized alternating matrix of corank one.  Then the following statements hold:
\begin{enumerate}
\item The associate smoothing diagram has a cycle of $n$ smoothable edges if and only if, up to permutations of the rows/columns, $\ps$ is a scalar multiple of $\FO{n}{k}$ for some integer $0 < k < n$ such that $\gcd(n,k)=\gcd(n,k+1)=1$.
\item In this case, the smoothable weights are given by
\[
\boldsymbol{\theta}_i := -{\rm e}_i - {\rm e}_{i+k'+1} + {\rm e}_{i+1} + {\rm e}_{i+k'}
\]
for $i \in \bZ/n\bZ$, where $k' := \overline{k}^{-1}$ is the inverse of $k$ modulo $n$.
\end{enumerate}

\end{proposition}
Note that the cycle of smoothable edges is then as follows:
\[
0 \edge (k'+1) \edge 2(k'+1) \edge \cdots \edge (n-1)(k'+1) \edge 0
\]
Consequently, the generic multiplicatively alternating matrices with a full cycle of smoothable edges are given by entrywise exponentiation of a multiple of $\FO{n}{k}$ for some $k$, i.e.~they have the form
\[
q_{ij} = v^{\FO[ij]{n}{k}}
\]
where $v \in \bCx$ is a constant.  Any deformation of the resulting algebra $\sfA_q$ over $\bC[\varepsilon]/(\varepsilon^2)$ is then isomorphic to one of the form
\begin{equation}\label{eq:CycleFirstOrderDef}
x_j x_i -  q_{ji} x_ix_j =  \begin{cases}
\varepsilon \gamma_i x_{i+1}x_{i+k'} , & j-i =  (k'+1)\mod n \\
-\varepsilon q_{ji} \gamma_j x_{j+1}x_{j+k'} , & j-i = - (k'+1)\mod n \\
0, & \mathrm{otherwise} \\
\end{cases}
\end{equation}
for constants $\gamma_i \in \bC$.  We claim that such deformations are unobstructed, and can indeed be extended to deformations that are algebraic functions on a smooth curve.

If at least one $\gamma_i$ is zero, then we are deforming along a cycle-free collection of smoothable edges, and \autoref{thm:quantumTorelliNoCycles} applies.  Otherwise, by rescaling the variables, we may arrange that all $\gamma_i$s are equal, so there is effectively one free parameter.  We establish the unobstructedness in this case by showing that the relations \eqref{eq:CycleFirstOrderDef} can be realized by degeneration of the celebrated elliptic algebras of Feigin--Odesskii, whose definition we now recall.

Let $\tau \in \mathbb{C}$, ${\rm Im}(\tau)>0$ be a parameter defining an elliptic curve $E=\mathbb{C}/\Lambda$, where $\Lambda= \mathbb{Z}+\tau\mathbb{Z}$ is the associated lattice.  Let $\varepsilon = \exp\rbrac{\tfrac{2 \pi \sqrt{-1}\tau}{n}}$; thus $\varepsilon^{n}$ is the square of the classical nome. Recall that the \emph{theta functions} of degree $n$ are defined by
$$
\theta_j (\Ept) = \theta(\Ept+\tfrac{j}{\n}\tau) \theta(\Ept+\tfrac{1}{\n}+\tfrac{j}{\n}\tau)~...~\theta(\Ept+\tfrac{\n-1}{\n}+\tfrac{j}{\n}\tau)  \exp\rbrac{2\pi \sqrt{-1}\rbrac{j \Ept +\tfrac{j(j-\n)}{2\n} \tau + \tfrac{j}{2\n} }},
$$
for $0 \le j < n$, where
$$
\theta(\Ept) = \sum_{k\in \mathbb{Z}} (-1)^k \exp\rbrac{2\pi \sqrt{-1} \rbrac{k\Ept + \tfrac{k(k-1)}{2} \tau }}
$$

\begin{definition}[Feigin--Odesskii~\cite{Feigin1998,Odesskii1989,Odesskii2002}]
For $0 < k < n$ such that $\gcd(\n,k)=1$, the \emph{Feigin-Odesskii algebra $Q_{n,k}(\varepsilon,z)$} is the quotient $\mathbb{C}\langle x_i:i\in \mathbb{Z}/n\mathbb{Z}\rangle$ by the quadratic relations
\begin{equation}\label{eq:FO_explicit_formula} \mathsf{rel}_{ij}^{nk} = \mathsf{rel}_{ij}^{nk}(\varepsilon,\Ept) =
\sum_{r\in \mathbb{Z}/n\mathbb{Z}} 
\dfrac{\theta_{j-i+r(k-1)}(0)}{\theta_{kr}(\Ept)
\theta_{j-i-r}(-\Ept)} x_{j-r}x_{i+r}
\end{equation}
for all  $i,j\in\mathbb{Z}/n\mathbb{Z}$.
\end{definition}
\begin{remark}
Note that the right hand side of \eqref{eq:FO_explicit_formula} depends \emph{a priori} on the parameter $\tau$.  However, it is invariant under translations of the form $\tau\mapsto \tau+n$ so that the relations depend only on $\varepsilon$, i.e.~are well-defined.
\end{remark}

It is known that $Q_{n,k}(\varepsilon,\Ept)$ has the Hilbert series of the polynomial ring in $n$ variables provided that $k$ is equal to one~\cite{Artin1990,Smith1992,Tate1996} or that $z+\Lambda \in \bC/\Lambda$ is not a nonzero torsion point of the elliptic curve~\cite{Chirvasitu2021}.  This is expected to hold for all $(n,k,\varepsilon,z)$.

The limit $\varepsilon \to 0$, or equivalently, $\tau \to \sqrt{-1}\infty$, corresponds geometrically to a degeneration of the elliptic curve to a singular curve of arithmetic genus zero, e.g.~a cycle of rational curves.  This limit recovers the first order deformations along a cycle as above:
\begin{proposition}\label{prop:FODeformOfToric}
 The Feigin--Odesskii relations $\mathsf{rel}^{nk}_{ij}(\varepsilon,z)=0$ from \eqref{eq:FO_explicit_formula} reduce modulo $\varepsilon^2$ to relations of the form \eqref{eq:CycleFirstOrderDef}, where $q_{ij}=v^{\FO[ij]{n}{k}}$, $v = \exp(2 \pi i z)$, and all coefficients $\gamma_i$ are nonzero.
\end{proposition}

\begin{proof}
Everywhere below, the notation $o(\varepsilon)$ denotes a quantity negligible compared to $\varepsilon$ as $\varepsilon\to 0$.
Note that in the definition of $\theta(\Ept)$, all the terms in the series tend to $0$ as $\varepsilon\to 0$, except the terms $k=0$ and $k=1$, so that
$$
\theta(\Ept) = 1 - \exp(2\pi \ii \Ept) + o(\varepsilon).
$$
Therefore
$$
\theta_j(\Ept) = \exp \rbrac{2\pi \sqrt{-1}\rbrac{j \Ept +\tfrac{j(j-\n)}{2\n} \tau + \tfrac{j}{2\n}}} \prod_{l=0}^{\n-1} \left(1 - \exp\rbrac{2\pi\ii\rbrac{\Ept+\tfrac{l}{\n}}}\,\varepsilon^{j} + o(\varepsilon)\right)
$$
for all $0 \le j < n$.
Note that for $\zeta$ a primitive $n$-th root of $1$, we have $\prod_{j=0}^{n-1} (1-t \zeta^j) =  1- t^n$. Therefore, we have
\begin{equation}\label{eq:thetaAlphaAsympt}
\theta_j(\Ept) = \begin{cases}
\exp\rbrac{2\pi \ii\rbrac{j \Ept +\frac{j(j-\n)}{2\n} \tau + \frac{j}{2\n}}} \big(1 + o(\varepsilon) \big)& 0 < j < n \\
\Big(1- \exp\rbrac{2\pi \ii n \Ept}\Big) \big(1 + o(\varepsilon) \big)& j=0.
\end{cases}
\end{equation}
which implies that
\begin{equation}\label{eq:FOCoeffAsympt}
\dfrac{\theta_{j-i+r(k-1)}(0)}{\theta_{kr}(\Ept)
\theta_{j-i-r}(-\Ept)} =
\begin{cases}
0 &  \overline{j-i+r(k-1)}=0 \\
\dfrac{v^{\overline{j-i}}}{1-v^n}(1 + o(\varepsilon)) &  \overline{kr}=0 \neq \overline{j-i+r(k-1)} \\
-\dfrac{v^{n-\overline{kr}}}{1-v^n} (1 +o(\varepsilon)) & \overline{j-i-r}=0\neq \overline{j-i+r(k-1)} \\
\pm \varepsilon^{g(kr,j-i-r)}v^{\overline{kr}-\overline{j-i-r}}(1 + o(\varepsilon)), & \overline{kr}, \overline{j-i-r}, \overline{j-i+r(k-1)}\neq 0
\end{cases}
\end{equation}
where $g(j,l) := f(j)+f(l) - f(j+l)$, with $f(j) := \frac{1}{2}\overline{j}(n-\overline{j})$, and the sign $\pm$ is $+$ if $\overline{kr}+\overline{j-i-r}<n$ and $-$ otherwise.
 The following properties of the function $g(j,l)$ are easily verified:
\begin{enumerate}
\item\label{it:pos_h_N} $g(j,l)$ is a positive integer for all integers $j,l$ not divisible by $\n$.
\item $g(j-1,l)<g(j,l)$, if $\overline{j}+\overline{l}\le \n$.
\item $g(j+1,l)<g(j,l)$, if $\overline{j}+\overline{l}\ge \n$.
\item\label{it:min_h_N} $g(j,l)=1$ if and only if either $\overline{j}=\overline{l}=1$, or $\overline{j}=\overline{l}=\n-1$.
\end{enumerate}
Thus \eqref{eq:FOCoeffAsympt} implies that for $\varepsilon=0$ the relations \ref{eq:FO_explicit_formula} reduce to
\begin{equation}\label{eq:FOOrderZeroRels}
\dfrac{v^{\overline{j-i}}}{1-v^n} x_{j}x_i - \dfrac{v^{n-\overline{k(j-i)}}}{1-v^n} x_i x_j = 0, ~~~0\le i\not=j <n
\end{equation}
which, up to scaling, are the same as $x_jx_i - v^{\FO[ji]{n}{k}}x_ix_j=0$ for all  $0\le i\not=j <n$.  Meanwhile, by property (4) of $f$, the only nontrivial terms of order $\varepsilon$ are those for which $\overline{r}=\overline{k'}$ and $\overline{j}=\overline{i+k'+1}$, or $\overline{r}=\overline{-k'}$ and $\overline{j}=\overline{i-k'-1}$, where $k'$ is the inverse of $k$ mod $n$ as above. In this way we arrive at the relations of the form \eqref{eq:CycleFirstOrderDef} with all $\gamma_i$ nonzero, as desired.
\end{proof}

\begin{remark}
The condition $\gcd(n,k+1)=1$ has not been used in the proof of \autoref{prop:FODeformOfToric}. This suggests the possibility of generalizing  \autoref{thm:Main} beyond the case in which ${\rm corank}(\lambda)=1$.
\end{remark}

\subsection{General deformations}
Finally, we treat the case of a deformation given by an arbitrary smoothing diagram.  We will do so by ``decoupling'' the contributions of each cycle from the rest of the diagram, using the following result; again, the proof is elementary but somewhat involved, so we delay it to the end of the section.
\begin{lemma}\label{lem:subcycles}
Let $\ps$ be a normalized alternating matrix of corank one, and let $I \subset \{0,\ldots,n-1\}$ be a subset forming the vertices of a cycle in the smoothing diagram of $\ps$.  Then the following statements hold:
\begin{enumerate}
\item All smoothable edges inside (resp.~outside) $I$ have their colored angles also inside (resp.~outside) $I$.
\item The submatrix  $(\ps_{ij})_{i,j\in I}$ is normalized of corank one, and its smoothing diagram consists of the given cycle $I$, with the same colorings of the edges and angles.
\item We have the identity
\begin{align*}
\ps_{i_1j}=\ps_{i_2j} \qquad \textrm{ for all }i_1,i_2 \in I \not \ni j.
\end{align*}
\end{enumerate}
\end{lemma}

Now let $\ps$ be any generic normalized alternating matrix of corank one, and let 
\[
\Sigma = \{i_\alpha\edge j_\alpha\}_{\alpha=1}^A
\]
be a set of smoothable edges.  If $I \subset \{0,...,n-1\}$, we denote by $\Sigma_I \subset \Sigma$ the set of edges in $\Sigma$ whose vertices lie in $I$.  Then there is a partition
\[
\{0,\ldots,n-1\} = J \sqcup I_1 \sqcup I_1\cdots \sqcup I_m,
\]
unique up to re-indexing the components $I_1,\ldots,I_m$,  such that $\Sigma_{I_s}$ is a cycle for all $1 \leq s \leq m$, and $\Sigma_J$ is cycle-free.  Let us denote the length of this cycle by
\[
n_s := |I_s|
\]
\autoref{lem:subcycles} part (2) and \autoref{prop:cycle-classification} imply that for each $s$, the sub-matrix $(\ps_{ij})_{i,j \in I_s}$ is, up to reindexing and rescaling, of the form $\FO{n_s}{k_s}$ for some $k_s$.

For $q = \EE(\ps)$, the $q$-symmetric algebra $\sfA_q = \mathsf{S}_q(\sfV^*)$ has the form
$$
\mathsf{S}_q(\mathsf{V}^*)=\mathsf{S}_{q}(\mathsf{V}_J^*) \otimes_q \mathsf{S}_{q}(\mathsf{V}_{I_1}^*) \otimes_q ... \otimes_q \mathsf{S}_{q}(\mathsf{V}_{I_m}^*) ,
$$
where $\sfV^*_J \subset \mathsf{V}^*$ is the subspace spanned by $x_i$, $i\in J$ and similarly for $\sfV_{I_s}$, and $\otimes_q$ denotes the tensor product with braiding given by the matrix $q$. 

We now deform the cycle-free factor $\mathsf{S}_q(\sfV^*_J)$, by taking the algebra $\mathsf{A}_{q,\Sigma_J}$ constructed in \autoref{thm:quantumTorelliNoCycles}, and considering its subalgebra $\sfA^J_{q,\Sigma_J,\varepsilon_0}$  generated by $x_i, i \in J$. Similarly, we deform the cycle factors $\mathsf{S}_q(\sfV^*_{I_s})$ to the corresponding Feigin--Odesskii algebras. We obtain the following

\begin{theorem}\label{thm:DeformCyclesAndChains}
The algebra
$$
\sfA_{q,\varepsilon_0,\ldots,\varepsilon_m}^\Sigma := \sfA_{q,\Sigma_J,\varepsilon_0}^J \otimes_q Q_{n_1,k_1}(\varepsilon_1,\Ept) \otimes_q ... \otimes_q Q_{n_m,k_m}(\varepsilon_m,\Ept),
$$
for $\varepsilon_0,\ldots,\varepsilon_m \in \bC$ is a quadratic algebra having the Hilbert series of a polynomial  ring whenever each of its factors does.  In particular, it gives a flat deformation of $\sfA_q$ over $\bC[\![\varepsilon_0,\ldots,\varepsilon_m]\!]$.  Thus non-toric deformations of $\sfA_q$ are jointly unobstructed.
\end{theorem}

\begin{proof} 
Once the flatness is established, the statement about unobstructedness is immediate, since the smoothable edges index a basis for the space of non-toric infinitesimal deformations and every non-toric infinitesimal deformation is equivalent, via the scaling action of $(\mathbb{C}^\times)^n$, to one where the edges within the same cycle have identical coefficients.

By \autoref{lem:subcycles} part (1), the deformed relations for each factor involve only the variable of each factor, so we only need to show that the braided tensor product $\otimes_q$ remains correctly defined after the deformation, i.e.~we need to show that the relations of the factors generate a two-sided ideal for the braided tensor product.

For $a,b \in J$, consider the relation $r_{ba} = - x_bx_a +q_{ba}x_ax_b + \sum_{m,k,\ell} \varepsilon_0^m C_{ba}^{k\ell} x_kx_\ell$ of the form \eqref{eq:deformedRelationsNoCycles} for the factor $\sfA_{q,\Sigma_J,\varepsilon_0}^J$.  We claim that for each $c \notin J$, we have $x_c r_{ba} = q_{cb} q_{ca} r_{ba} x_c$. Indeed, for each term $x_k x_\ell$ appearing in $r_{ba}$ the torus weight $-\mathrm{e}_a -\mathrm{e}_b + \mathrm{e}_k + \mathrm{e}_\ell$ is a linear combination of smoothable weights. Since each such a smoothable weight ${\bf w}$ is Hochschild-contributing, we have $\prod_{i=0}^{n-1} q_{ci}^{w_i} = 1$. This implies $q_{cb}q_{ca} = q_{ck}q_{c\ell}$. Therefore, $x_c r_{ba} = q_{cb} q_{ca} r_{ba} x_c$ as desired.

Similarly, suppose that $1 \leq s \leq m$ and $c \notin I_s$.  Then by \autoref{lem:subcycles} part (3), there exists a constant $q_c$ such that $x_cx_a = q_c x_a x_c$ for all $a \in I_s$.  Consequently, if $r$ is any quadratic relation for $Q_{n_s,k_s}(\varepsilon_s,\Ept)$, we have $x_c r = q_c^2 r x_c$.
\end{proof}

\begin{proof}[Proof of \autoref{lem:subcycles}]
Without loss of generality, we may assume that $I = \{0,\ldots,m-1\}$ and the cycle of edges is given by
\[
0 \edge 1 \edge \cdots \edge m-1 \edge 0.
\]
Let $b$ be the biresidue matrix of $\ps$.  By \autoref{lm:pos_rat_valency} part (2), we may assume that the entries $b_{i,i+1}$, $i\in I$, as well as $b_{m-1,0}$, are all positive reals.  We claim that 
\begin{align}
b_{i_1j}=b_{i_2j} \qquad \textrm{ for all }i_1,i_2 \in I \not \ni j  \label{eq:b-cycle-condition}
\end{align}
Indeed, smoothability of the edges $i\edge(i+1)$ for $0 \leq i \leq m-2$, and $(m-1)\edge 0$, implies that for each $j\notin I$ we have the cycle of inequalities $b_{0j} \le b_{1j} \le b_{2j}\le ... \le b_{m-1,j} \le b_{0j}$; hence all of these inequalities must be equalities, as desired.  

From \eqref{eq:b-cycle-condition}, we deduce statement (1), as follows.  Firstly, if an angle $i_1 \edge j \edge i_2$ is colored, then $b_{i_1j}+b_{ji_2}$ is nonzero, but \eqref{eq:b-cycle-condition} and $b$ being alternating ensure that this cannot happen if $i_1,i_2 \in I \not\ni j$, so the  smoothable edges inside $I$ have their colored angles also inside $I$.  On the other hand, suppose that  $j_1\edge j_2$, $j_1,j_2\notin I$ is a smoothable edge outside $I$.  We claim that no angle of the form $j_1\edge i \edge j_2$ with $i \in I$ can be colored.  Indeed, by \eqref{eq:b-cycle-condition}, the numbers $b_{j_1i},b_{j_2,i}$ are independent of $i$, so if one such angle is colored, they all must be.  But it impossible for them all to be colored, because the cycle has length $m \ge 3$, and a given smoothable edge cannot have more than two colored angles.  Hence none of them are colored, as claimed.
 
Now let $b' = (b_{ij})_{i,j \in I}$ be the submatrix indexed by the cycle.  We claim that $b'$ is normalized of corank one.  Indeed, since $b$ is normalized, we have for every $i \in I$ that the $i$th row sum of $b'$ is given by
$$
\sum_{j\in I} b_{ij} = - \sum_{j\notin I} b_{ij}.
$$
Applying \eqref{eq:b-cycle-condition} to the right hand side, we deduce that the row sum of $b'$ is equal to a constant $C$ that is independent of $i$.  Averaging over all $i\in I$ and using the fact that $b$ is alternating, we get
\[
0 = \tfrac{1}{m}\sum_{i,j \in I}b_{ij} = - C,
\]
and hence $b'$ is normalized.  To see that the corank is one,  let $\Delta=\{(z_0,...,z_{n-1})\in\mathbb{C}^n:z_0+...+z_{n-1}=0\}$, and let $\bC^I \subset \bC^n$ be the subspace spanned be the basis vectors labeled by $I$.   Then, by definition, $b:\Delta\to \Delta$ is the inverse of $\ps:\Delta\to \Delta$.  Moreover \eqref{eq:b-cycle-condition} implies that $b(\Delta\cap \mathbb{C}^I)=\Delta \cap \mathbb{C}^I$. Since $b:\Delta\to \Delta$ is invertible, its restriction to the invariant subspace $\Delta\cap \mathbb{C}^I$ is invertible, too. This implies that $b'$ has corank one.

Considering the matrix entries, we deduce that $b'$ is the biresidue matrix of the submatrix $\ps' = (\ps_{ij})_{i,j\in I}$ and hence statement (2) follows.  Considering the action of $\ps$ on the vectors ${\rm e}_{i_1}-{\rm e}_{i_2}$, $i_1,i_2\in I$, we use   $\ps(\Delta\cap \mathbb{C}^I)=\Delta\cap \mathbb{C}^I$ to obtain statement (3).
\end{proof}

\section{Homological properties of the deformed algebras}

 Recall that if $q$ is a multiplicatively alternating matrix, the corresponding algebra $\mathsf{A}_q=\mathsf{S}_q(\mathsf{V}^*)$ is Koszul (\autoref{lm:qPolyKoszul}), and is furthermore Calabi--Yau if and only if $q$ is normalized (\autoref{lm:qPolyCY}). We now show that these properties hold also for the deformed algebras constructed from smoothing diagrams in \autoref{thm:DeformCyclesAndChains} in the previous section (except possibly for countably  many values of the deformation parameter in the presence of cycles).  Since the properties of being Koszul or twisted Calabi--Yau are preserved by braided tensor products of graded algebras, the problem reduces to the case of cycle-free deformations, and the Feigin--Odesskii algebras.  For the latter, the following is known:
\begin{theorem}[Chirvasitu--Kanda--Smith~\cite{Chirvasitu2023}]
For fixed $\varepsilon$, the Feigin--Odesskii algebras $Q_{n,k}(\varepsilon,\Ept)$ are Koszul and twisted Calabi-Yau  for all but at most countably many $[z] \in \bC/\Lambda$.
\end{theorem}
We therefore focus on the cycle-free case from now on.  We thus fix a biresidue matrix $b$ inverse to a generic normalized alternating matrix $\ps$, and a subset $I \subset \{0,\ldots,n-2\}$ defining a cycle-free collection of smoothable edges $i \edge (i+1)$, $i \in I$ as in \autoref{sec:nocycles}. We will establish the following.
\begin{theorem}
Let $\sfA_{q,I}$ be a cycle-free deformation of $\sfA_q$ as in \autoref{thm:quantumTorelliNoCycles}.  Then $\sfA_{q,I}$ is Koszul and Calabi--Yau of dimension $n$.
\end{theorem}

\subsection{Cycle-free deformations are Koszul}\label{sec:Koszul}

To establish the Koszul property, we use the notion of a Gr\"{o}bner basis for a two-sided ideal $J$ of the free non-commutative algebra $\mathsf{T}(\mathsf{V}^*)=\mathbb{C}\langle x_0,...,x_{n-1}\rangle$, which we now recall.

Choose a well-ordering $<$ on the set of monomials in $x_0,...,x_{n-1}$ that is compatible with multiplication, in the sense that $t_1<t_2$ implies $lt_1r<lt_2r$ for any monomials $r,l,t_1,t_2$. For an element $f\in \mathbb{C}\langle x_0,...,x_{n-1}\rangle$, we denote by $LT(f)$ be the leading term of $f$ with respect to $<$. We say that a monomial $t_1$ is divisible by a monomial $t_2$, or that $t_2$ divides $t_1$, if $t_1=lt_2r$ for some monomials $l,r$. Given a two-sided ideal $J$ of $\mathsf{T}(\mathsf{V}^*)$, a subset $G\subset J$ is called a \textit{Gr\"{o}bner basis} for $J$ with respect to $<$ if for every $f\in J$, the leading term $LT(f)$ is divisible by $LT(g)$, for some $g\in G$.

Let $i\edge (i+1)$, $i\in I$ be a cycle-free set of smoothable edges for a biresidue matrix $b$. Let $\mathsf{A}_{q,I}=\mathsf{T}(\mathsf{V}^*)/J_{q,I}$ be the filtered deformation of $\mathsf{A}_q=\mathsf{S}_q(\mathsf{V}^*)$ constructed in \autoref{thm:quantumTorelliNoCycles}. The two-sided ideal $J_{q,I}$ is generated by the quadratic elements
\begin{equation}
\label{eq:deformedRelsNoCycle}
r_{ba}(\varepsilon) = -x_b x_a + q_{ba} x_a x_b + \sum_{m=1}^N \varepsilon^m \left(\sum_{\substack{-\mathrm{e}_a-\mathrm{e}_b+\mathrm{e}_k+\mathrm{e}_\ell \in \Theta_m\\k\le \ell}} C_{ba}^{k\ell} x_k x_\ell\right),~~~a<b.
\end{equation}
We define a well-ordering $<$ on the monomials in $\mathsf{T}(\mathsf{V}^*)$ using the smoothable filtration $\mathcal{F}_\bullet$ defined in \eqref{eq:smoothableFiltration}, as follows.
For two monomials $t_1,t_2\in\mathsf{T}(\mathsf{V}^*)$, we say that $t_1<t_2$ if

\begin{itemize}
\item[-] $\max(a:t_1\in \mathcal{F}_a) > \max(a:t_2\in\mathcal{F}_a)$, or
\item[-] $\max(a:t_1\in \mathcal{F}_a) = \max(a:t_2\in\mathcal{F}_a)$ and $t_1$ is smaller than $t_2$ degree-lexicographically.
\end{itemize}

\begin{proposition}\label{prop:deformedKoszulNoCycle}
The following statements hold: 
\begin{enumerate}
\item The relations $r_{ba}(\varepsilon)$, $b>a$, in \eqref{eq:deformedRelsNoCycle} form a Gr\"{o}bner basis  for $J_{q,I}$ with respect to the ordering $<$.

\item The algebra $\mathsf{A}_{q,I}$ is Koszul.
\end{enumerate}
\end{proposition}

\begin{proof}
(1) Given any $f\in J_{q,I}$, let us prove that $LT(f)$ is divisible by $x_bx_a=-LT(r_{ba}(\varepsilon))$, for some $b>a$. Choose $a\ge0$ such that $f\in\mathcal{F}_a\setminus \mathcal{F}_{a+1}$. We can uniquely express $f=f_1+f_2$ so that each term in $f_1$ belongs to $\mathcal{F}_a\setminus \mathcal{F}_{a+1}$ and $f_2\in \mathcal{F}_{a+1}$. The choice of ordering guarantees $LT(f)=LT(f_1)$. Since $\mathsf{A}_{q,I}$ is a filtered deformation of $\mathsf{A}_q$, the element $f_1$ belongs to the two-sided ideal generated by $r_{ba}(0)$, $a<b$. Hence $LT(f_1)$ is divisible by some $x_bx_a$, $b>a$. 

(2) The Gr\"{o}bner basis $r_{ba}(\varepsilon)$, $b>a$, consists of quadratic elements. Hence the claim follows from \cite[Theorem 3.iii)]{Green1995}
\end{proof}

\subsection{Cycle-free deformations are Calabi--Yau}\label{sec:CY}

To establish the Calabi--Yau property, we use the theory of superpotentials.  Recall that a \textit{superpotential} of degree $n$ on a vector space $\sfV$ is an element $\Phi \in (\mathsf{V}^*)^{\otimes n}$ that is invariant under the action of $\mathbb{Z}/n\mathbb{Z}$ generated by
\begin{equation}\label{eq:ZnAction}
x_{i_1} \otimes x_{i_2} \otimes ... \otimes x_{i_n} \longmapsto (-1)^{n-1} x_{i_2} \otimes \cdots \otimes x_{i_n} \otimes x_{i_1}.
\end{equation}
Given a superpotential $\Phi$ of degree $n$, one can construct a quadratic algebra $\mathsf{A}=\mathcal{D}(\Phi)$ with the relations
$$
\partial_{x_{i_1}} ... \partial_{x_{i_{n-2}}} \Phi, ~~~~i_1,...,i_{n-2} \in \{0,...,n-1\},
$$
where the operation $\partial_{x_j}$ is defined by
$$
\partial_{x_j} \big( x_{i_1} \otimes x_{i_2} \otimes ... \otimes x_{i_m} \big) = \left\{
\begin{matrix}
x_{i_2} \otimes ... \otimes x_{i_m}, & \text{if } j=i_1 \\
0,& \text{otherwise}.
\end{matrix}
\right.
$$
The  $\mathsf{A}$-bimodule $\mathsf{A}$ fits into into the following complex of projective $\mathsf{A}$-bimodules
$$
\mathcal{W}^\bullet(\Phi):= 0 \rightarrow \mathsf{A}\otimes W_n\otimes \mathsf{A} \xrightarrow{d_n} \mathsf{A}\otimes W_{n-1}\otimes \mathsf{A} \rightarrow ... \rightarrow \mathsf{A}\otimes W_1\otimes \mathsf{A} \xrightarrow{d_1} \mathsf{A}\otimes W_0\otimes \mathsf{A} \rightarrow 0,
$$
where $W_m\subset (\mathsf{V}^*)^{\otimes i}$ is the $\mathbb{C}$-span of $\partial_{x_{i_1}} ... \partial_{x_{i_{n-m}}} \Phi$, $i_1,...,i_{n-m}\in \{0,...,n-1\}$, and
$$
d_m = \varepsilon_m (\mathsf{split}_L + (-1)^m \mathsf{split}_R),
$$
$$
\mathsf{split}_L(a\otimes x_{i_1} x_{i_2}...x_{i_m} \otimes a') = ax_{i_1}\otimes x_{i_2}...x_{i_m} \otimes a',~~~a,a'\in\mathsf{A},
$$
$$
\mathsf{split}_R(a\otimes x_{i_1} x_{i_2}...x_{i_m} \otimes a') = a\otimes x_{i_1} ... x_{i_{m-1}} \otimes x_{i_m} a',~~~a,a'\in\mathsf{A},
$$
$$
\varepsilon_m = 
\left\{
\begin{matrix}
(-1)^{m(n-m)}, & \text{if } m<(n+1)/2 \\
1, & \text{otherwise}.
\end{matrix}
\right.
$$

\begin{theorem}[{Bocklandt--Schedler--Wemyss \cite[Theorem 6.2]{Bocklandt2010}}]
\label{thm:BSW}
A quadratic algebra $\mathsf{A}=\mathsf{T}(\mathsf{V}^*)/J$ is Koszul and Calabi--Yau if and only if it is of the form $\mathcal{D}(\Phi)$ for a superpotential $\Phi$, and the corresponding complex $\mathcal{W}^\bullet(\Phi)$ is a resolution of $\sfA$.
\end{theorem}

For a normalized, multiplicatively alternating matrix $q=(q_{ij})_{i,j=0}^{n-1}$, the algebra $\mathsf{S}_q(\mathsf{V}^*)$ can be obtained as $\mathcal{D}(\Phi_0)$ for 
$$
\Phi_0 = x_0 \wedge_q x_1 \wedge_q ... \wedge_q x_{n-1} = \sum_{\sigma\in\Sigma_n} \left(\prod_{\substack{i<j:\\ \sigma_i>\sigma_j}} -q_{\sigma_j\sigma_i}\right) x_{\sigma_0} \otimes x_{\sigma_1} \otimes ... \otimes x_{\sigma_{n-1}}.
$$

Let $i\edge (i+1)$, $i\in I$ be a cycle-free set of smoothable edges for a biresidue matrix $b$. Let $\mathsf{A}_{q,I}$ be the filtered deformation of $\mathsf{A}_q=\mathsf{S}_q(\mathsf{V}^*)$ constructed in \autoref{thm:quantumTorelliNoCycles}. For $\varepsilon\in\mathbb{C}$, let $\mathsf{A}_{q,I,\varepsilon}$ be the algebra with relations \eqref{eq:deformedRelationsNoCycles} (by construction, each $\mathsf{A}_{q,I,\varepsilon}$, $\varepsilon\not=0$, is isomorphic to $\mathsf{A}_{q,I}$).

\begin{proposition}\label{prop:deformedCYNoCycle} 
The algebra $\mathsf{A}_{q,I}$ is $n$-Calabi--Yau and Artin--Schelter regular.
\end{proposition}

\begin{proof} By \cite[Lemma 1.2]{Rogalski2014}, the Artin--Schelter regularity follows from the Calabi-Yau property. Let us show the latter.  By \autoref{prop:deformedKoszulNoCycle}, the algebra $\mathsf{A}_{q,I,\varepsilon}$ is Koszul for each $\varepsilon\in\mathbb{C}$. The filtered deformation  $\mathsf{A}_{q,I,\varepsilon}$ of  $\mathsf{A}_{q}$ induces a filtered deformation of the Koszul resolution. In particular, the top degree Koszul syzygy $(\mathsf{A}_{q,I,\varepsilon})^{\ac}_n\subset(\mathsf{V}^*)^{\otimes n}$ remains one-dimensional for all $\varepsilon$. We choose $\Phi_\varepsilon$ to  the unique basis vector for $(\mathsf{A}_{q,I,\varepsilon})^{\ac}_n$ such that $\Phi_\varepsilon$ is a filtered deformation of $\Phi_0$. 
It follows that the complex $\mathcal{W}^\bullet(\Phi_\varepsilon)$ remains a resolution of $\mathsf{A}_{q,I,\varepsilon}$ for all $\varepsilon$ and $\mathsf{A}_{q,I,\varepsilon} = \mathcal{D}(\Phi_\varepsilon)$. By \cite[Lemma 6.5]{Bocklandt2010}, the complex $\mathcal{W}^\bullet(\Phi_\varepsilon)$ is a subcomplex of the Koszul resolution of $\mathsf{A}_{q,I,\varepsilon}$, and so, having the same dimensions, they have to coincide. To apply \autoref{thm:BSW}, it remains to check that $\Phi_\varepsilon$ is invariant under the action of $\mathbb{Z}/n\mathbb{Z}$ generated by \eqref{eq:ZnAction}.

To this end, let us write the superpotential $\Phi_\varepsilon\in(\mathsf{V}^*)^{\otimes n}$ as
$$
\Phi_\varepsilon = \sum_{i=0}^{n-1} x_i \otimes v_i,
$$
where $v_0, ..., v_{n-1}$ is a basis of ${W}_{n-1}(\Phi_\varepsilon) = (\mathsf{A}_{q,I,\varepsilon})^{\ac}_{n-1}$. Then we also have
$$
\Phi_\varepsilon = (-1)^{n-1} \sum_{i=0}^{n-1} v_i \otimes x'_i,
$$
for some $x'_0, ...., x'_{n-1}\in\mathsf{V}^*$.  We must show that $x_i' = x_i$ for all $i$, or equivalently that the  ``twisting'' of $\Phi_\varepsilon$, defined as  the linear map 
\[
\mapdef{\tau_\varepsilon}{\sfV^*}{\sfV^*}{x_i}{x_i'}
\]
is the identity.  Note that $\tau_\varepsilon$ extends uniquely to a tensor product preserving map $\tau_\varepsilon:T(\mathsf{V}^*)\to T(\mathsf{V}^*)$. One checks that $\tau_\varepsilon$ is a filtered deformation of the identity and preserves the quadratic relations of $\mathsf{A}_{q,I,\varepsilon}$, thus giving a filtered unipotent, graded automorphism of $\mathsf{A}_{q,I,\varepsilon}$. By \autoref{lm:filteredAuto} below, the twisting $\tau_\varepsilon$ must be the identity, as desired.
\end{proof}

\begin{lemma}\label{lm:filteredAuto}
Let $T_\varepsilon\in {\rm Aut}(\mathsf{A}_{q,I,\varepsilon})$, $\varepsilon\in\mathbb{C}$, be a polynomial family of filtered unipotent, graded automorphisms such that $T_0$ is identity. Then $T_\varepsilon$ is identity for all $\varepsilon\in\mathbb{C}$.
\end{lemma}

\begin{proof}
If $T_\varepsilon$ is not identity, then there exists $x_a$ such that 
$$T_\varepsilon (x_a)= x_a + \varepsilon^k \sum_{d\not=a} \alpha_d x_d  + O(\varepsilon^{k+1}),$$ 
where at least one $\alpha_d$ is non-zero. Here and below the notation $O(\varepsilon^{k+1})$ denotes terms in which $\varepsilon$ appears with a power greater than or equal to $k+1$. Note in the sum that $d$ can only occur when $e_d - e_a$ is a sum of $\geq k$ smoothable weights corresponding to the subset $I$, and in particular this explains why $d \neq a$.

Without loss of generality, we assume that the index $a$ is chosen so that $k$ is the lowest possible. Fix an index $b\not=a$ such that $\alpha_{b}\not=0$.
We claim that we can choose $c\not=a,b$ such that $q_{ca}\not=q_{cb}$. Since $\prod_j q_{ij} = 1$ for all $j$, we note that if $q_{ab} \neq 1$, then already $\prod_{j \neq a,b} q_{ja} \neq \prod_{j \neq a,b} q_{jb}$, which implies the desired existence.  So assume that $q_{ab}=1$.  Then, if $q_{ca}=q_{cb}$ for all $c \neq a,b$, we have $q_{ca}=q_{cb}$ for all $c$, and then $e_a-e_b$ is a Hochschild-contributing weight of $\mathsf{A}_q$. This contradicts the observation in \autoref{subsec:smoothingDiagrams} that the only sum-zero weight of $\HP(\mathsf{B}_\ps)$ is zero.

Let $$T_\varepsilon(x_c) = x_c + \varepsilon^k \sum_{d\not=c} \beta_{d} x_d + O(\varepsilon^{k+1}).$$ Recalling that $r_{ac}(\varepsilon)=-x_cx_a+q_{ca}x_ax_c+O(\varepsilon)$, we obtain
$$
T_\varepsilon(r_{ac}(\varepsilon)) = r_{ac}(\varepsilon) + \varepsilon^k \Bigg( \sum_{d\not=a} \alpha_{d}(-x_cx_{d}+q_{ca}x_{d}x_c)+\sum_{d\not=c} \beta_d( -x_dx_a + q_{ca} x_a x_d) \Bigg)+ O(\varepsilon^{k+1}) =
$$
$$
= r_{ac}(\varepsilon) +\varepsilon^k \alpha_b(q_{ca}-q_{cb})x_bx_c +  \varepsilon^k \Bigg( \sum_{d\not=a,b} \alpha_{d}(q_{ca}-q_{cd})x_{d}x_c+\sum_{d\not=c} \beta_d (q_{ca}-q_{da}) x_a x_d \Bigg) + O(\varepsilon^{k+1}).
$$
Since $T_\varepsilon(r_{ac}(\varepsilon))\in J_{q,I}$, the $\varepsilon^k$ terms above should cancel each other out. However, the coefficient of the $x_bx_c$ term is non-zero, a contradiction.
\end{proof}

\section{Kontsevich's conjecture}\label{sec:kontsevich}

\subsection{Deformations with a formal parameter}
In order to apply our result to Kontsevich's conjecture, we need to consider algebras over the formal power series $\mathbb{C}[\![\hbar]\!]$ or convergent series $\bC\{\hbar\}$ instead of $\bC$.  It is straightforward to adapt our results to this setting.     The main result we need is the following consequence of \autoref{thm:quantumTorelliNoCycles}. 
\begin{corollary}\label{cor:convergenceDeformedAqNoCycles}
Let $\ps$ be a generic normalized alternating matrix, and consider the $\mathbb{C}[\![\hbar]\!]$-algebra $\mathsf{A}_q$, where  $q=\EE(\hbar \ps) \in \bC[\![\hbar]\!]$.  Then the following statement holds.
\begin{enumerate}
\item The deformation $\mathsf{A}_{q,I}$ from \autoref{thm:quantumTorelliNoCycles} can be defined over $\mathbb{C}\{\hbar\}$. Furthermore, if $\gamma_i=\hbar$, for each $i\in I$, then in the relations \eqref{eq:deformedRelationsNoCycles} one can arrange $C_{ba}^{k\ell}\in\hbar\mathbb{C}\{\hbar\}$.
\item For any flat deformation of $\mathsf{A}_q$ with quadratic relations \eqref{eq:deformedRelationsNoCycles} such that all $C_{ba}^{k\ell}\in \hbar \mathbb{C}[\![\hbar]\!]$, one can apply an automorphism of $\mathsf{V}[\![\hbar]\!]$ to make all $C_{ba}^{k\ell}\in \hbar \mathbb{C}\{\hbar\}$.
\end{enumerate}
\end{corollary}

\begin{proof}
It suffices to establish the convergence properties for the corresponding Maurer--Cartan elements $\sum_m \varepsilon^m \mu_m$ in the Hochschild complex used in the proof of \autoref{thm:quantumTorelliNoCycles}.

For (1), choose $\gamma_i=\hbar$, for each $i\in I$. Then in the proof of \autoref{thm:quantumTorelliNoCycles}, we inductively check that the Hochschild $3$-cocycle ${\rm Obs}_m$ of $\mathsf{A}_q$, viewed as $\mathbb{C}\{\hbar\}$-algebra, is divisible by $\hbar^2$. Therefore, by \autoref{lm:HHtoricPowerSeries} below, one can choose its primitive $2$-cochain $\mu_m$ to be divisible by $\hbar$.

For (2), without loss of generality, let us assume $\gamma_i=\hbar$, $i\in I$. Let $m\ge 2$ be the minimal index such that the $\varepsilon^m$-term $\mu_m$ is divergent in $\hbar$. Then ${\rm Obs}_m$ is convergent in $\hbar$ and, by part (1), it has a convergent primitive $\widetilde{\mu}_{m}$, in addition to the divergent one $\mu_{m}$, both primitives being divisible by $\hbar$. Since $\Theta_m$ does not contain Hochschild-contributing weights, by \autoref{lm:HHtoricPowerSeries} below, we can express $\widetilde{\mu}_m-\mu_m = \mathsf{d}(\eta)$ for some Hochschild $1$-cochain $\eta$ of the $\mathbb{C}[\![\hbar]\!]$-algebra $\mathsf{A}_q$.  By applying $\exp(\varepsilon^m\eta)$, we make $\widetilde{\mu}_m=\mu_m$, and continue by induction on $m$.
\end{proof}
This made use of the following vanishing result for Hochschild cohomology of $\mathsf{A}_q$:
\begin{lemma}\label{lm:HHtoricPowerSeries}
Let $\ps \in \askew{\bC}$ be an alternating matrix, let $q = \EE(\hbar\ps)\in \bC[\![\hbar]\!]$ and let $\sfA_q$ be the associated $q$-symmetric algebra over $\bC[\![\hbar]\!]$.  If $\mathbf{w} \in \bZ^n$ is not $q$-Hochschild-contributing, then the weight component $\mathsf{HH}(\mathsf{A}_q)^\mathbf{w}$ is annihilated by multiplication by $\hbar$.  The same is true when replacing $\bC[\![\hbar]\!]$ with $\bC\{\hbar\}$.
\end{lemma}

\begin{proof}
We repeat the proof of \autoref{lm:HHtoric}, viewing $\mathsf{A}_q$ as $\bC((\hbar))$-algebra.
For a non-Hochschild-contributing weight $\mathbf{w}$, there is at least one index $j$ with $w_j\ge 0$ such that the coefficient $\alpha_j$ in \eqref{eq:KoszulDiffAq} is of the form $\hbar f$, $f\in \bC[\![\hbar]\!]^\times$.  Then the homotopy operator \eqref{eq:homotopyKoszulDiff} applied to a cocycle in $ \Big(\mathsf{S}^\bullet(\mathsf{V}^*) \otimes \wedge^\bullet(\mathsf{V}) \otimes \hbar\bC[\![\hbar]\!]\Big)^\mathbf{w}$ gives its primitive in $ \Big(\mathsf{S}^\bullet(\mathsf{V}^*) \otimes \wedge^\bullet(\mathsf{V}) \otimes \bC[\![\hbar]\!]\Big)^\mathbf{w}$, where $\mathsf{V}\cong \bC^n$.  The same proof works for $\bC\{\hbar\}$.
\end{proof}

\subsection{The star product}
Let $\mathcal{O}_\mathsf{V}=\mathbb{C}[x_0,...,x_{n-1}]$. Recall that a \emph{star product} on $\mathcal{O}_\mathsf{V}$ is an associative product 
\begin{equation}\label{eq:starProduct}
\star: \mathcal{O}_\mathsf{V}[\![\hbar]\!] \times \mathcal{O}_\mathsf{V}[\![\hbar]\!] \longrightarrow \mathcal{O}_\mathsf{V}[\![\hbar]\!]
\end{equation}
such that $f \star g = fg + O(\hbar)$ for all $f,g\in \mathcal{O}_\mathsf{V}$.
Two star products $\star_1$ and $\star_2$ are called \emph{gauge-equivalent} if 
the two $\mathbb{C}[\![\hbar]\!]$-algebras $(\mathcal{O}_\mathsf{V}[\![\hbar]\!],\star_1)$ and $(\mathcal{O}_\mathsf{V}[\![\hbar]\!],\star_2)$ are isomorphic by an isomorphism that reduces to the identity on $\mathcal{O}_\mathsf{V} = \mathcal{O}_\mathsf{V}[\![\hbar]\!]/(\hbar)$.  A star product $\star$ is called \emph{convergent} if $f \star g \in \mathcal{O}_\mathsf{V}\{\hbar\}$ for all $f,g \in \mathcal{O}_\mathsf{V}\{\hbar\}$.

In \cite{Kontsevich2003}, Kontsevich proved that every Poisson structure on $\mathcal{O}_\mathsf{V}$  admits a canonical quantization, given by a star product whose commutator is equal to the Poisson bracket modulo $\hbar^2$.  When the Poisson bracket is quadratic, its canonical quantization is a quadratic $\bC[\![\hbar]\!]$-algebra, and Kontsevich has conjectured that the quadratic relations converge  \cite[Conjecture 1]{Kontsevich2001}.  We consider the following slightly weaker version where we allow for convergence ``up to isomorphism''.
\begin{conjecture}\label{conj:convergenceKontsevich}
Let $\sigma$ be a quadratic Poisson bivector on $\mathsf{V}=\mathbb{C}^n$. Then its canonical quantization is gauge-equivalent to a convergent star product.
\end{conjecture}

Recently the second named author and Lindberg proved Kontsevich's conjecture
for the toric Poisson algebras $\mathsf{B}_\ps$, where $\ps$ is any alternating matrix. Specifically, building upon Hodge-theoretic ideas proposed by Kontsevich \cite{Kontsevich2008}, they proved  
\begin{theorem}[\cite{Lindberg2021,Lindberg2024}]\label{thm:toricKontsevich}
  The generators $x_i$ of the canonical quantization of $\mathsf{B}_\ps$ satisfy the relations of $\mathsf{A}_q$. In particular, the former is isomorphic to the $\hbar$-adic completion of $\mathsf{A}_q$ for $q=\EE(\hbar\ps)$.
\end{theorem}
It is expected that similar techniques can be applied to the Feigin--Odesskii algebra.

The goal of this section is to leverage \autoref{thm:toricKontsevich} to prove the convergence of the canonical quantization for Poisson structures obtained from smoothing diagrams without cycles.   We thus fix a corank-one normalized alternating matrix $\ps$, with associated toric Poisson algebra $\sfB_\ps$, and let $i\edge (i+1)$, $i\in I$ be a cycle-free subset of smoothable edges.

We have the following Poisson analogue of \autoref{thm:quantumTorelliNoCycles}.

\begin{theorem}\label{thm:PoissonTorelliNoCycles}
For every choice of non-zero constants $\gamma_i$, $i\in I$, the algebra $\mathsf{B}_\ps$ admits a filtered deformation to a Poisson algebra $\mathsf{B}_{\ps,I}$, whose Poisson bracket is defined by
\begin{equation}\label{eq:deformedBracketNoCycles}
\{x_a,x_b\} = \ps_{ab} ~x_a x_b + \sum_{m=1}^N \varepsilon^m \left(\sum_{\substack{-\mathrm{e}_a-\mathrm{e}_b+\mathrm{e}_k+\mathrm{e}_\ell \in \Theta_m\\k\le \ell}} c_{ba}^{k\ell} ~x_k x_\ell\right),~~~a<b,
\end{equation}
where $c_{ba}^{k\ell}=\gamma_i$ whenever $-\mathrm{e}_a-\mathrm{e}_b+\mathrm{e}_k+\mathrm{e}_\ell =\boldsymbol{\theta}_i$.
The algebra $\mathsf{B}_{\ps,I}$ is uniquely defined, up to a filtered automorphism of $\mathsf{V}$, and independent of $\varepsilon$ or the choice of non-zero constants $\gamma_i$.
\end{theorem}

\begin{proof}
This is essentially a special case of the main results of \cite{Matviichuk2020}, for which the full machinery of \emph{op.\ cit.} is not needed. Indeed, a direct proof can be obtained by repeating the proof of \autoref{thm:quantumTorelliNoCycles}, replacing the Hochschild cohomology with Poisson cohomology, and Gerstenhaber bracket with Schouten bracket. The existence claim also follows directly from \autoref{thm:quantumTorelliNoCycles} by taking semiclassical limit.
\end{proof}

Now, we are ready to state and prove the main result of this section.

\begin{theorem}\label{thm:kontsevich}
The following statements hold
\begin{enumerate}
\item The canonical quantization of $\mathsf{B}_{\ps,I}$ is isomorphic to the $\mathbb{C}[\![\hbar]\!]$-algebra $\mathsf{A}_{q,I}$ constructed in \autoref{cor:convergenceDeformedAqNoCycles}, where $q=\EE(\hbar \ps)$.
\item \autoref{conj:convergenceKontsevich} is true for $\mathsf{B}_{\ps,I}$.
\end{enumerate}
\end{theorem}

\begin{proof}
  Since the canonical quantization procedure is $(\mathbb{C}^\times)^n$-equivariant, the weights of the terms in the canonical star product of $\mathsf{B}_{\ps,I}$ have to be the same as the weights appearing in \eqref{eq:deformedBracketNoCycles}. Using \autoref{thm:toricKontsevich}, we deduce that the star product satisfies relations of the form \eqref{eq:deformedRelationsNoCycles} with $q = \EE(\hbar\ps)$. Moreover, by \autoref{cor:convergenceDeformedAqNoCycles}, all $C_{ab}^{k\ell}$ appearing in the relations \eqref{eq:deformedRelationsNoCycles} belong to $\hbar\mathbb{C}[\![\hbar]\!]$, and can be defined over convergent power series.  Up to gauge equivalence, we can furthermore assume that monomials are given by $q$-symmetrization:
  \[
    x_{i_1} \cdots x_{i_m} = \frac{\sum_{\sigma \in S_m} q^{c_\sigma} x_{i_{\sigma(1)}} \star \cdots \star x_{i_{\sigma(m)}} }{\sum_{\sigma \in S_m} q^{c_\sigma}},
  \]
  where $c_{\sigma}$ is  such that, in $\mathsf{A}_q$, we have $q^{c_\sigma} x_{i_{\sigma(1)}} \cdots x_{i_{\sigma(m)}} = x_{i_1} \cdots x_{i_m}$.  With this choice (or actually, any other choice of gauge which identifies monomials with $\bC\{\hbar\}$-linear combinations of star products of the $x_i$), it follows that all of the star products converge.  
\end{proof}

\bibliographystyle{hyperamsplain}
\bibliography{DeformOfToricQuant}
\end{document}